\documentclass[12pt,notitlepage,twoside]{article}

\usepackage{latexsym}
\usepackage{amsmath}
\usepackage{amsfonts}
\usepackage{amssymb}
\let \n = \noindent
\let \dis = \displaystyle

\newcommand{\be}{\begin{equation}}
\newcommand{\ee}{\end{equation}}

\newtheorem{thm}{Theorem}[section]
\newtheorem{prop}[thm]{Proposition}
\newtheorem{lem}[thm]{Lemma}

\newtheorem{cor}[thm]{Corollary}

\numberwithin{equation}{section}

\author{Wael Abdelhedi$^\dag$,  Hichem Chtioui$^\dag$ and Hichem Hajaiej$^\ddag$ \footnote{ E-mail
 addresses: \texttt{wael\_hed@yahoo.fr} ( W. Abdelhedi), \texttt{Hichem.Chtioui@fss.rnu.tn} (H. Chtioui) and \texttt{hichem.hajaiej@gmail.com.}(H. Hajaiej).}\\
{\footnotesize $^\dag$Department of mathematics,}\\ {\footnotesize
Faculty of Sciences of Sfax, 3018 Sfax,
Tunisia.}\\ {\footnotesize $^\ddag$
NYU Shanghai.}}

\title{A  complete study  of  the lack of compactness  and  existence  results  of a  Fractional Nirenberg Equation via a  flatness hypothesis: Part I}%
\begin{document}

\date{ }

\maketitle

{\footnotesize

\n{\bf Abstract.} In this paper, we consider a nonlinear critical problem involving the fractional  Laplacian operator arising in conformal geometry, namely the prescribed  $\sigma$-curvature problem on the standard  n-sphere $n\geq 2$. Under the assumption that the prescribed function is flat  near its critical points, we give precise estimates on the losses of the compactness and we provide existence results. In this first part, we will focus on the case $1<\beta\leq n-2\sigma$, which was not included in the results of Jin, Li and   Xiong \cite{jlx1} and \cite{jlx2}.\\
\n{\bf  MSC 2000:}\quad  35J60, 35B33, 35B99, 35R11, 58E30.\\
\n {\bf Key words:} Fractional Laplacian, critical exponent, $\sigma$-curvature, critical points at infinity.}




\section{Introduction and main results}

Fractional calculus has attracted a lot of scientists during the last decades. This is essentially due to its numerous applications in various domains: Medicine, modeling populations, biology, earthquakes, optics, signal processing, astrophysics, water waves, porous media, nonlocal diffusion, image reconstruction problems; see \cite{hajaiej1} and references [1, 2, 6, 7, 13, 14, 19, 22, 25, 36, 38, 41, 43, 45, 46, 58] therein.

Many important properties of the Laplacian are not inherited, or are only partially satisfied, by its fractional powers. This gave birth to many challenging and rich mathematical problems. However, the literature remained quite  silent until the publication of the breakthrough paper of Caffarelli and Silvester in 2007, \cite{cas1}. This seminal work has hugely contributed to unblock a lot of difficult problems and opened the way for the resolution of many other ones. In this paper, we study another important fractional PDE whose resolution also requires some novelties because of the nonlocal properties  of the operator present in it. More precisely,  we investigate the existence of  solutions for the following critical fractional nonlinear equation

\begin{equation}\label{1.3}
    P_{\sigma}u = c(n, \sigma) K u^{\frac{n+2\sigma}{n-2\sigma}},
    \qquad u > 0, \hbox{ on } \; S^{n}.
\end{equation}
where $\sigma\in (0, 1)$,  $K$ is  a positive function  defined on $(S^n, g)$,   $$P_{\sigma}=\frac{\Gamma(B+\frac{1}{2} +\sigma)}{\Gamma(B+\frac{1}{2} -\sigma)}, \qquad B= \sqrt{-\Delta_g+\Big(\frac{n-1}{2}\Big)^2},$$
$\Gamma$ is the Gamma function, $c(n, \sigma) = \Gamma (\frac{n}{2}+\sigma)/ \Gamma (\frac{n}{2}-\sigma)$, and $\Delta_g$ is the Laplace-Beltrami operator on $(S^n, g)$. The operator $P_{\sigma}$
can be seen more concretely on $\mathbb{R}^n$ using stereographic projection. The stereographic projection from $S^n\setminus\{N\}$ to $\mathbb{R}^n$ is the inverse of $F: \mathbb{R}^n\rightarrow S^n\setminus\{N\}$ defined by $$F(x)=\Big(\frac{2x}{1+|x|^2}, \frac{|x|^2-1}{|x|^2+1}\Big)$$where $N$ is the north pole of $S^n$. For all $f\in C^\infty(S^n)$, we have
\begin{equation}\label{}
 (P_{\sigma}(f))\circ F= \Big(\frac{2}{1+|x|^2}\Big)^\frac{-(n+2\sigma)}{2} \Big(-\Delta\Big)^{\sigma}\Big(\Big(\frac{2}{1+|x|^2}\Big)^\frac{n-2\sigma}{2} (f\circ F) \Big)
\end{equation}
where $\big(-\Delta\big)^{\sigma}$ is the fractional Laplacian operator (see, e.g., page 117 of \cite{Stein}).

Problem (\ref{1.3}) is heavily connected to the fractional order curvature, usually called the $\sigma$-curvature. This challenging problem has been first addressed in \cite{jlx1} and \cite{jlx2}. In these two seminal papers, the authors have been able to show the existence of solutions of (\ref{1.3}) and to derive some compactness properties. More precisely, thanks to a very subtle approach based on approximation of the solutions of (\ref{1.3}) by a blowing-up subcritical method, they proved the existence of solutions for the critical fractional Nirenberg problem (\ref{1.3}), (see Theorem 1.1 and Theorem 1.2 of \cite{jlx1}). Their method is based on tricky variational tools, in particular they have established many interesting fractional functional inequalities. Their main hypothesis is the so-called flatness condition:\\
Let $K: \mathbb{S}^{n}\rightarrow\mathbb{R}$, be a $C^2$ positive function. We say that $K$ satisfies a flatness condition $(f)_{\beta}$: if for each critical point $y$ of $K$ there exist $(b_i)_{i\leq n}\in\mathbb{R}^*$, such that in some geodesic normal coordinate centered at $y$, we have
\begin{equation}\label{hf}
K(x) = K(y) + \sum_{i=1}^{n}b_i|(x-y)_i|^\beta + R(x-y),
\end{equation}where $ b_i=b_i(y)\in\mathbb{R}^*, \sum_{i=1}^{n}b_i\neq 0$ and $\sum_{s=0}^{[\beta]}|\nabla^{s}R(y)||y|^{-\beta-s}=o(1)$ a $y$ tends to zero. Here $\nabla^s$ denotes all possible derivatives of order $s$ and $[\beta]$ is the integer part of $\beta$.

However, they have only been able to handle the case $n-2\sigma< \beta< n$ in the flatness hypothesis. This excludes some very interesting functions $K$. In fact, note that an important class of functions which is worth to include in any results of existence for (\ref{1.3}) are the Morse functions ($C^2$ having only non-degenerate critical points). Such functions can be written in the form $(f)_{\beta}$ with $\beta=2$. Since Jin, Li and Xiong require $n-2\sigma< \beta< n$ ($0<\sigma<1$),  their theorems do not apply to this relevant class of functions. Moreover, they require some additional technical assumptions ($K$ antipodally symmetric in Theorem 1.1 and $K\in C^{1, 1}$ positive in Theorem 1.2).

Motivated by the breakthrough papers \cite{jlx1} and \cite{jlx2} and aiming to include a larger class of functions $K$ in the existence results for (\ref{1.3}), we develop in this paper a self-contained approach which enables us to include all the plausible cases ($1<\beta<n$). Our method hinges on a readapted characterization of critical points at infinity techniques of the proof are different for $1< \beta \leq n-2\sigma$ and $n-2\sigma\leq\beta <n$. In this work, we will handle the first case.

The spirit of this approach goes back to the work of Bahri \cite{b1} and Bahri-Coron \cite{BC1}. Nevertheless, the nonlocal properties of the fractional Laplacian involve many additional obstacles and require some novelties in the proof. Note that in \cite{ac1}, the two first authors have given   an existence result for $n=2$,\; $0<\sigma <1$, through an Euler-Hopf type formula. In their paper, they  assumed that $K$ is a Morse function satisfying  the following non-degeneracy condition:\\

\n $$(nd) \qquad  \Delta K(y)\neq 0 \; \mbox{ whenever  } \;\nabla K(y)=0.$$
We point out that the criterium of \cite{ac1} has an equivalent in dimension three (see \cite{ac4}). However,  the method cannot be generalized to higher dimensions $n\geq 4$ under the condition (nd), since   the corresponding
index-Counting-Criteria, when taking into account all the critical points at infinity
is always equal to 1.

Convinced that the non-degeneracy assumption would exclude some interesting class of functions $K$, we opted for the flatness hypothesis used in \cite{jlx1} and \cite{jlx2}. But again, in order to include all plausible cases (both $1<\beta\leq n-2\sigma$ and $n-2\sigma\leq \beta < n$), we need to develop a new line of attack with new ideas. This leads to an interesting new phenomenon; that is the presence of multiple blow-up points. In fact,  looking to the possible formations of blow-up points, it turns  out that the strong interaction of the bubbles in the case where $n-2\sigma<\beta <n$ forces all blow-up points to be single, while in the case where $1<\beta <n-2\sigma$ such an  interaction of two bubbles is negligible with respect to the self interaction, while if $\beta =n-2\sigma$ there is a phenomenon of balance that is the interaction of two bubbles is of the same order with respect to the self interaction.  In order to state our results, we need the following notations and assumptions. Let

\begin{center}$\mathcal{K}=  \{ y \in S^{n}, \nabla K(y) = 0\}$
\end{center}
\begin{center}$\mathcal{K}^{+}=  \{ y \in \mathcal{K}, -  \sum_{k=1}^{n} b_{k}(y)> 0  \}$\end{center}
\begin{center}
$\widetilde{i}(y) = \sharp\Big\{ b_{k}=b_{k}(y),\; 1 \leq k \leq n
\mbox{ such that } b_{k}< 0 \Big \}.$
\end{center}
\begin{center}$\mathcal K_{n-2\sigma}=\big\{y\in \mathcal K,\; \beta= \beta(y)=n-2\sigma\big\}.$\end{center}
For each p-tuple,  $p\geq1$  of distinct points $\tau_{p}:=(
y_{l_{1}},..., y_{l_{p}})$ $\in(\mathcal K _{n-2\sigma})^{p}$, we define a
p$\times$p symmetric matrix $ M(\tau_{p})=(m_{ij})$ by
\begin{equation}\label{4}
    m_{ii} =\frac{n-2\sigma}{n} \widetilde {c}_{1}\frac{- \sum_{k=1}^{n}
   b_{k}\big(y_{l_{i}}\big)}{K\big(y_{l_{i}}\big)^{\frac{n}{2\sigma}
    }}, \hskip 0.5 cm   m_{ij} =  2^{\frac{n-2\sigma}{2}}c_{1}\frac{-G\big( y_{l_{i}}, y_{l_{j}}\big)}
  {\big[K\big( y_{l_{i}}\big)  K(
  y_{l_{j}}\big)\big]^{\frac{n-2\sigma}{4\sigma}}},
\end{equation}
where
\begin{equation}\label{G}
G( y_{l_{i}},  y_{l_{j}})= \dis\frac{1}{(1-\cos d( y_{l_{i}},  y_{l_{j}}))^\frac{n-2\sigma}{2}}
\end{equation}
$$c_{1}=\int_{ \mathbb
R^{n}}\frac{ d\displaystyle
x}{\big(1+\big|x\big|^{2}\big)^{\frac{n+2\sigma}{2}}}  \; \mbox{ and   } \; \widetilde
{c}_{1}= \displaystyle
\int_{\displaystyle \mathbb
R^{n}}\frac{|x_{1}|^{n-2}}{\big(1+|x|^{2}\big)^{n}}dx.$$
Here $x_{1}$ is the first
component of $x$ in some geodesic normal coordinates system. Let $\rho(\tau_{p})$
be the least eigenvalue of $ M\big(\tau_{p}\big)$.\\
 $\big(A_{1}\big)$ \; \;
Assume that $\rho\big (\tau_{p}\big)\neq0$ for each distinct
points $y_{{1}},...,y_{{p}}\in\mathcal K _{n-2\sigma}$.\\
Now, we introduce the following sets:
$$
\mathcal C_{n-2\sigma}^{\infty}:=\big\{\tau_{p}=(y_{l_{1}},...,y_{l_{p}})\in (\mathcal K_{n-2\sigma})^{p}, p\geq1, s.t.\hskip0.2cm y_{i}\neq y_{j }\hskip0.2cm \forall i\neq j, \mbox{ and }
 \rho(\tau_{p})>0\big\},
$$
$$
\mathcal {C}_{<(n-2\sigma)}^{\infty}:=\big\{\tau_{p}=(y_{l_{1}},...,y_{l_{p}})\in ( \mathcal{ K}^{+}\backslash \mathcal{K}_{n-2\sigma})^{p}, p\geq1, s.t.\hskip0.2cm y_{i}\neq y_{j }\hskip0.2cm \forall i\neq j \big\}.
$$
For any $\tau_{p}=(y_{l_{1}},...,y_{l_{p}})\in (\mathcal K)^{p}$, we denote $i(\tau_{p})_{\infty}= p-1+ \dis \sum_{j=1}^{p}\,[\, n\,-\, \widetilde{i}\,(y_{l_{j}})\,]$.

The main result of this paper is the following.

\begin{thm}\label{TH2}
Assume that $K$ satisfies $(A_{1})$ and $(f)_{\beta}$, with $1<\beta\leq n-2\sigma.$ If
$$\dis\sum_{\dis\tau_{p}\in \mathcal C_{n-2\sigma}^{\infty}} (-1)^{\dis i(\tau_{p})_{\infty}} + \dis\sum_{\dis\tau'_{p}\in \mathcal {C}_{<(n-2\sigma)}^{\infty}} (-1)^{\dis i(\tau'_{p})_{\infty}}\qquad \qquad \qquad$$$$\qquad \qquad \qquad - \dis\sum_{\dis(\tau_{p},\tau'_{p})\in \mathcal C_{n-2\sigma}^{\infty}\times {C}_{<(n-2\sigma)}^{\infty}} (-1)^{\dis i(\tau_{p})_{\infty}+i(\tau'_{p})_{\infty}}\neq1, $$
then (\ref{1.3}) has at least one solution.
\end{thm}

In part 2, we  will  address  the case $n-2\sigma\leq \beta < n$, following  another approach and  recovering  the  main results  of \cite{jlx1} and \cite{jlx2}.  More precisely, we will prove:

\begin{thm}\label{TH1}
Assume that $K$ satisfies $(A_{1})$ and $(\mathfrak{f})_{\beta}$, with $n-2\sigma\leq \beta < n.$ If\\
$$ \hskip3cm \dis\sum_{\dis y\in \mathcal K^{+}\backslash\mathcal K_{n-2\sigma}} (-1)^{\dis i(y)_{\infty}}+ \dis\sum_{\dis\tau_{p}\in \mathcal C_{n-2\sigma}^{\infty}} (-1)^{\dis i(\tau_{p})_{\infty}}\neq1  \hskip6cm$$
then (\ref{1.3}) has at least one  solution.
\end{thm}

\smallskip We organize the remainder of our  paper as follows. The second
section is devoted to recall some preliminary results ralated to the Caffarelli-Silvestre method  (see \cite{cas1}). In section three,    we characterize   the
critical points at infinity  of the associated variational
problem.  In the fourth   section,   we give  the proof of the  main results. The characterization of critical points at infinity  requires some technical results which for the convenience
of the reader, are given in the appendix.

\section{Preliminary results}

In this section, we recall some preliminary results ralated to the Caffarelli-Silvestre extension   (see \cite{cas1}), which provides a variational structure to the fractional problem.

\n We say that  $u\in H^{\sigma}(S^n)$ is a solution of (\ref{1.3}) if the identity
\begin{equation}\label{}
 \int_{S^n}  P_{\sigma}u \varphi dx= c(n, \sigma) \int_{S^n} K u^{\frac{n+2\sigma}{n-2\sigma}} \varphi dx,
\end{equation}
holds for all $\varphi\in H^{\sigma}(S^n)$, where
$H^\sigma(S^n)=\{u\in L^2(S^n), \|u\|^2_{H^\sigma(S^n)}\in  L^2(S^n)\},$
equipped with the norm,
\begin{equation}\label{}
\|u\|_{H^\sigma(S^n)}= \Big(\int_{S^n} P_{\sigma}u  u \Big)^{1/2}.
\end{equation}

\n We recall that the set of smooth functions $C^\infty(S^n)$ is dense in $H^{\sigma}(S^n)$. Observe also that for  $u\in H^{\sigma}(S^n)$, we have $u^{\frac{n+2\sigma}{n-2\sigma}}\in L^\frac{2n}{n+2\sigma}(S^n)\hookrightarrow H^{-\sigma}(S^n)$.

We associate to problem (\ref{1.3}), the functional
\begin{equation}\label{}
    I(u)=\frac{1}{2}\int_{S^n}u P_{\sigma}u -\frac{n-2\sigma}{2n}\int_{S^n} K u^\frac{2n}{n-2\sigma},
\end{equation}
defined in $H^{\sigma}(S^n)$.

Motivated by the work of Caffarelli and Silvestre \cite{cas1}, several authors have considered an equivalent
definition of the operator $P_\sigma$  by
means of an auxiliary variable, see \cite{cas1}, (see also \cite{bcps1}, \cite{cabs1}, \cite{cabt1}, \cite{cdds1} and \cite{sto1}). In fact, we handle  problem (\ref{1.3}), through a localization method introduced by Caffarelli
and Silvestre  on the Euclidean space $\mathbb{R}^n$, through which
(\ref{1.3}) is connected to a degenerate elliptic differential equation in one dimension higher by a Dirichlet to Neumann map. This
provides a good variational structure to the problem.  By studying
this problem with classical local techniques, we establish existence of
positive solutions. Here the Sobolev trace embedding
comes into play, and its critical exponent $2^*=\frac{2n}{n-2\sigma}$.

\n Namely, let $D_n= S^n\times[0, \infty)$. Given $u\in H^{\sigma}(S^n)$,  we define its harmonic extension $U=  E_\sigma(u)$ to $D_n$ as the solution to the problem
\begin{equation}\label{eq:4}
\begin{cases}
     - div (t^{1-2\sigma}\nabla U )&=  0 \text{ in
  } D_n\\
U &= u
     \text{ on  }  S^{n}\times\{t=0\}.
\end{cases}
\end{equation}
The extension belongs to the space $H^1(D_n)$ defined as the completion of $C^\infty(D_n)$ with the norm
\begin{equation}\label{}
    \|U\|_{H^1(D_n)}= \Big(\int_{D_n}t^{1-2\sigma}|\nabla U|^2 dxdt\Big)^{1/2}.
\end{equation}
Observe that this extension is an isometry in the sense that
\begin{equation}\label{}
    \|E_\sigma(u)\|_{H^1(D_n)}= \|u\|_{H^{\sigma}(S^n)}, \qquad \forall u\in H^{\sigma}(S^n).
\end{equation}
Moreover, for any $\varphi\in H^1(D_n)$, we have the following trace inequality
\begin{equation}\label{}
  \|\varphi\|_{H^1(D_n)} \geq \|\varphi(., 0)\|_{H^{\sigma}(S^n)}.
\end{equation}
\n The relevance of the extension function $U=E_\sigma(u)$ is that it is related to the fractional Laplacian of the
original function $u$ through the formula

\begin{equation}\label{}
    \displaystyle  \dis -\lim_{t\rightarrow 0^+}t^{1-2\sigma}\frac{\partial U}{ \partial t}(x, t)= P_\sigma u(x).
\end{equation}

\n Thus, we can reformulate (\ref{1.3}) to the following

\begin{equation}\label{eq:5}
\begin{cases}
     div (t^{1-2\sigma}\nabla U(x, t))&=0 \qquad  \qquad \qquad \text{ and } \, U>0 \text{ in
  } D_n\\
     \dis -\lim_{t\rightarrow 0^+}t^{1-2\sigma}\frac{\partial U}{ \partial t}(x, t)&=
      K U^{\frac{n+2\sigma}{n-2\sigma}}(x, 0)
     \qquad  \;\quad\text{ on  }  S^{n}\times\{0\}.
\end{cases}
\end{equation}

\n The functional associated to (\ref{eq:5}), is given by

\begin{equation}\label{}
    I_1(U)=\frac{1}{2}\int_{D_n}t^{1-2\sigma}|\nabla U|^2dxdt -\frac{n-2\sigma}{2n}\int_{S^n} K U^\frac{2n}{n-2\sigma}dx,
\end{equation}
defined in $H^1(D_n)$.

\n Note  that critical points of $I_1$ in  $H^1(D_n)$ correspond to critical points of $I$ in  $H^\sigma(S^n)$. That is, if $U$ satisfies (\ref{eq:5}), then the trace
$u$ on $S^n\times{0}$ of the function $U$ will be a solution of problem (\ref{1.3}). Let also define the functional
\begin{equation}\label{}
        J(U)=\frac{\|U\|^2_{H^1(D_n)}}{\bigg(\dis\int_{S^n} K U^\frac{2n}{n-2\sigma}dx\bigg)^{\frac{n-2\sigma}{n}}},
\end{equation}
defined on  $\Sigma$ the unit sphere of
 $H^{1}(D_n)$. We set, $\Sigma^+ = \{U \in \Sigma/ \,
 U \geq 0 \}$. Problem (\ref{1.3}) will be reduced to  finding the critical points of $J$ under  the
constraint $U\in \Sigma^+$. The exponent $\frac{2n}{n-2\sigma}$ is critical for the Sobolev
trace embedding $H^{1}(D_n) \rightarrow L^q(
\mathbb{S}^{n})$. This embedding is continuous and not compact. The
functional $J$ does not satisfy the Palais-Smale condition, which leads to the failure
of the standard critical point theory. This  means that there exists a
 sequence $(u_n)$ belonging to the constraint such that  $J(u_n)$ is bounded, its gradient goes to zero
 and  does  not converge.  The analysis of sequences failing PS
condition can be analyzed along the ideas introduced in
 \cite{BC1} and \cite{S3}.

 In order to describe such a
 characterization in our case, we need to introduce some notations.

\n For $a\in \partial\mathbb{R}^{n+1}_+ $ and $\lambda
>0$, define the function:
$$ \tilde{\delta}_{a, \lambda}(x) =
\bar{c}\dis\frac{\lambda^\frac{n-2\sigma}{2}}{\Big((1+ \lambda
 x_{n+1})^2
+ \lambda^2|x'-a'|^2\Big)^\frac{n-2\sigma}{2}}$$ \n where $x \in
\mathbb{R}^{n+1}_+$, and $\bar{c}$ is chosen such that $\tilde{\delta}_{a,
\lambda}$ satisfies the following equation,

$$\begin{cases}
    \; \; \; \Delta U &= 0 \quad\text{and } u>0 \text{ in  }
\mathbb{R}^{n+1}_+\\
     \displaystyle-\frac{\partial U}{\partial x_{n+1}} &=
u^{\frac{n+2\sigma}{n-2\sigma}}
     \!\!\qquad\qquad\text{ on  } \partial \mathbb{R}^{n+1}_+.
\end{cases}$$

\n Set $$ \delta_{a, \lambda} = i^{-1}(\tilde{\delta}_{a,
\lambda}).$$
where $i$ is an isometry from $H^1(D^n)$ to $ D^{1, 2}(\mathbb{R}^{n+1}_+)$.\\ In the sequel, we will identify $\delta_{a, \lambda}$ and its composition with $i$. We will also identify the function  $u$  and its   extension $U$. These facts will be assumed as understood in the sequel.

\n For $\varepsilon>0$, $p\in \mathbb{N}^*$, we define

$$V(p, \varepsilon)= \begin{cases}
u\in \Sigma \text{ s. t }\exists a_1, \ldots, a_p \in {S}^{n}, \exists  \alpha_1, \ldots, \alpha_p >0
 \text{ and }\\
 \exists \lambda_1, \ldots, \lambda_p > \varepsilon^{-1} \text{
with  } \Big\| u - \dis\sum_{i=1}^p\alpha_i\delta_{a_i, \lambda_i}
 \Big\|< \varepsilon, \;\; \varepsilon_{ij} < \varepsilon\; \forall
 \;i\neq j, \\
    \text{ and }  \Big|
   J(u)^\frac{n}{n-2\sigma}\alpha_i^\frac{2}{n-2\sigma} K(a_i) - 1
      \Big| < \varepsilon \;\; \forall i, j=1, \ldots , p,
      \end{cases}$$
where $$\varepsilon_{ij} =
\Bigg(\dis\frac{\lambda_i}{\lambda_j} +
\frac{\lambda_j}{\lambda_i} + \lambda_i\lambda_j |a_i-
a_j|^2\Bigg)^\frac{2\sigma-n}{2}.$$

\section{Characterization of the critical points at infinity for $1<\beta\leq n-2\sigma$}

This section is devoted to the characterization of the critical points at infinity in $V(p,\varepsilon), p\geq1$, under $\beta$-flatness condition with $1<\beta\leq n-2\sigma$. This characterization is obtained through the construction of a suitable pseudo-gradient at infinity for which the Palais-Smale condition is satisfied along the decreasing flow-lines as long as these flow-lines do not enter in the neighborhood of finite number of critical points $y_{i}, i=1,...,p$
of $K$ such that $$(y_{1},...,y_{p})\in\mathcal P^{\infty}:= {C}_{<(n-2\sigma)}^{\infty}\cup  C^{\infty}_{n-2\sigma} \cup {C}_{<(n-2\sigma)}^{\infty}\times C^{\infty}_{n-2\sigma}.$$ More precisely we have:

\begin{thm}\label{1}
Assume that $K$ satisfies $(A_{1})$ and $(f)_{\beta}$, $1 <\beta \leq n-2\sigma$. \\Let
$\beta:=\max\{ \beta(y)/ y \in \mathcal K\} $. For $p \geq 1$,
there exists a pseudo-gradient $W$ in $V(p,
\varepsilon)$ so that the following holds.\\
There exist a constant $c > 0$ independent of $u= \dis
\sum_{i=1}^{p}\alpha_{i} \delta_{i} \in V(p,
\varepsilon)$ such that
$$ (i)\Big\langle  \partial J(u), W(u) \Big\rangle \leq -c \biggr( \sum_{i=1}^{p}\dis
  \frac{1}{\lambda_{i}^{\beta}}+ \sum_{i=1}^{p} \dis \frac{\mid\nabla
   K(a_{i})\mid}{\lambda_{i}} +\dis \sum_{j \neq i}
   \varepsilon_{ij} \biggl).\hskip4.7cm$$
 $$(ii)\Big\langle  \partial J(u+\overline{v}), W(u)+ \dis
 \frac{\partial\overline{v}}{\partial(\alpha_{i},a_{i},\lambda_{i})}(W(u))
  \Big\rangle \leq -c \biggr( \sum_{i=1}^{p}\dis
  \frac{1}{\lambda_{i}^{\beta}}+ \sum_{i=1}^{p} \dis \frac{\mid\nabla
  K(a_{i})\mid}{\lambda_{i}} +\dis \sum_{j \neq i}
   \varepsilon_{ij} \biggl).  $$
Furthermore $| W|$ is bounded and the only case
where the maximum of the $\lambda_{i}$'s is not bounded is when
$a_{i}\in B(y_{l_{i}}, \rho)$ with $y_{l_{i}}\in \mathcal{K},$
$\forall i=1,...,p$, $(y_{l_{1}},...,y_{l_{p}})\in \mathcal P^{\infty}$.
\end{thm}

In order to prove  Theorem \ref{1}, we  state the following two
results which deal with two specific cases of Theorem \ref{1}.
Let,
$$V_{1}(p,\varepsilon)= \Big\{ u= \dis \sum_{i=1}^{p} \alpha_{i} \delta_{i} \in V(p, \varepsilon) \mbox{ s.t },
a_{i} \in B(y_{l_{i}}, \rho), y_{l_{i}}\in \mathcal{K}\setminus \mathcal{K}_{n-2\sigma} \; \forall i=1,...,p \Big\}.$$
$$ V_{2}(p,\varepsilon)=\Big\{ u= \dis \sum_{i=1}^{p}
\alpha_{i} \delta_{i} \in V(p, \varepsilon) \mbox{ s.t }, a_{i} \in
B(y_{l_{i}}, \rho),  y_{l_{i}}\in \mathcal{K}_{n-2\sigma} , \; \forall
i=1,...,p \Big\}.$$
We then have:

\begin{prop}\label{3.2}
For $p\geq 1$, there exist a pseudo-gradient $W_{1}$ in
$V_{1}(p,\varepsilon)$ such that the  following holds:\\
Theres exist $c> 0$ independent of $u=\dis \sum_{i=1}^{p}
\alpha_{i} \delta_{i} \in V_{1}(p, \varepsilon)$ such that
$$\Big \langle  \partial J(u), W_{1}(u)\Big\rangle \leq  - c  \biggl(\dis \sum_{i=1}^{p} \dis\frac{1}{\lambda_{i}^{\beta}}
+ \dis \sum_{i\neq j} \varepsilon_{i j}+ \dis \sum_{i=1}^{p}
\dis\frac{|\nabla K(a_{i})|}{\lambda_{i}}    \biggr). $$
Furthermore $|W_{1}|$ is bounded in $H^1(D^n)$ and the only case
where the maximum of the $\lambda_{i}$'s is not bounded is when
$a_{i}\in B\,(y_{l_{i}}\,,\ \rho)$ with $y_{l_{i}}\in \mathcal{K}^{+}$,
$\forall i=1,...,p$,  with
$(y_{l_{1}},...,y_{l_{p}}) \in \mathcal{C}^{\infty}_{<n-2\sigma}$.
\end{prop}

\begin{prop}\label{3.1}
For $p  \geq 1$ there exists a pseudo-gradient $W_{2}$ in
$V_{2}(p,\varepsilon)$ such that $\forall  u = \dis
\sum_{i=1}^{p}\alpha_{i} \delta_{i} \in V_{2}(p, \varepsilon) $,
we have $$\Big \langle  \partial J(u), W_{2}(u)\Big\rangle \leq -
c \biggl(\dis \sum_{i=1}^{p} \dis\frac{1}{\lambda_{i}^{n-2}} +
\dis \sum_{i\neq j} \varepsilon_{i j}+ \dis \sum_{i=1}^{p}
\dis\frac{|\nabla K(a_{i})|}{\lambda_{i}}    \biggr). $$
Where $c$ is a positive constant independent of $u$. Furthermore,
we have $|W_{2} |$ is bounded and the only case where the maximum
of $\lambda_{i}'$s is not bounded is when $a_{i}\in B(y_{l_{i}},
\rho)\;, y_{l_{i}}\in \mathcal{K}^{+},\;\; \forall \;i=1,...,p$, with
$(y_{l_{1}},...,y_{l_{p}}) \in \mathcal{C}^{\infty}_{n-2\sigma}$
\end{prop}

\n In our construction of the pseudogradien $W$, we will use the following notations. \\
Let $u=\dis
\sum_{i=1}^{p} \alpha_{i}
\delta_{i}\in
V(p,\varepsilon)$, such that $a_{i}\in B(y_{l_i},\rho)$, $y_{l_{i}} \in \mathcal{K}$, $\forall i=1,...,p$. \\For
simplicity, if $a_{i}$ is close to a critical point $y_{l_{i}}$, we
will assume that the critical point is at the origin, so we will confuse
$a_{i}$ with $(a_{i}-y_{l_{i}})$.  Now, let $i \in \{1,...,p\}$
and let $M_{1}$ be a positive large constant. We will say that
$$ i \in L_{1} \mbox{   if   } \lambda_{i} |a_{i} | \leq M_{1}$$
and we will say that
$$  i \in L_{2} \mbox{   if   } \lambda_{i} |a_{i} | > M_{1}.$$
For each $i \in \{1,...,p  \}$, we define the following vector
fields:

\begin{equation}\label{ref1} Z_{i}(u)= \alpha_{i}\lambda_{i}
\dis \frac{ \partial \delta_{i}}{\partial \lambda_{i}}
\end{equation}

\begin{equation}\label{ref2}
 X_i= \alpha_{i}\sum_{k=1}^{n}\frac{1}{\lambda_i}\frac{\partial
\delta_{i}}{\partial(a_i)_k}\int_{\mathbb{R}^{n}}b_k\frac{|x_k+\lambda_i(a_i)_k|^\beta}
{(1+\lambda_i|(a_i)_k|)^{\beta-1}} \frac{x_k}{(1+|x|^2)^{n+1}}dx, \hskip2.4cm
\end{equation}
where $(a_i)_k$ is the $k^{th}$ component of $a_i$ in some
geodesic normal coordinates
system.\\
 We claim that $X_{i}$ is bounded. Indeed, the claim is
trivial if $i \in L_{1}$. If $i \in L_{2}$, by elementary
computation, we have the following estimate:

\begin{eqnarray}\nonumber
\dis \int_{\dis \mathbb{R}^{n}} \dis \frac{\big|x_{k} +
\lambda_{i} (a_{i})_{k} \big|^{\beta} x_{k}}{(1+ |
x|^{2})^{n+1}}dx &=&(\lambda_{i} |(a_{i})_{k} |)^{\beta} \dis
\int_{\dis \mathbb{R}^{n}} \big| 1+\frac{x_{k}}{\lambda_{i}
((a_{i})_{k})}\big|^{\beta}\frac{x_{k}}{(1+\mid x\mid^{2})^{n+1}} \dis dx\\
\label{eq4.1} &=& c(\mbox{sign} \lambda_{i}
(a_{i})_{k})(\lambda_{i} |(a_{i})_{k} |)^{\beta-1}(1+o(1)),
\end{eqnarray}
for any $k$, $1\leq k\leq n$ such that $\lambda_{i}|(a_{i})_{k}|>\frac{M_{1}}{\sqrt{n}}$. Hence our claim is valid.\\
Let $k_{i}$ be an index such that
\begin{eqnarray}\label{M}
|(a_{i})_{k_{i}}|= \dis \max_{1\leq j \leq n} |(a_{i})_{j} |.
\end{eqnarray}
It easy to see that if $i \in L_{2}$ then
$\lambda_{i}|(a_{i})_{k_{i}}|>\frac{M_{1}}{\sqrt{n}}$.

\vskip0.2cm
\noindent\textbf{Proof of  Theorem \ref{1}}
In order to complete the construction of the pseudo-gradient $W$ suggested in Theorem \ref{1}, it only remains (using proposition \ref{3.1} and \ref{3.2})
to focus attention at the two following subsets of $V(p,\varepsilon)$.\\
\underline{Subset 1}. We consider here the case of $u=\dis \sum_{i =1}^{p}\alpha_{i}\delta_{i} =\dis \sum_{i \in
I_{1}}\alpha_{i}\delta_{i}+ \dis \sum_{i \in
I_{2}}\alpha_{i}\delta_{i}$ such that
$$I_{1}\neq \emptyset,  I_{2}\neq \emptyset,\dis \sum_{i \in
I_{1}}\alpha_{i}\delta_{i}\in  V_{1}(\sharp I_{1},\varepsilon)\;\mbox{and} \dis \sum_{i \in
I_{2}}\alpha_{i}\delta_{i}\in  V_{2}(\sharp I_{2},\varepsilon). $$
Without loss of generality, we can assume that $ \lambda_{1}\leq \cdots\leq \lambda_{p}.$
We distinguish three cases. \\

\noindent\underline{case 1.}
$u_{1}:=\dis \sum_{i \in
I_{1}}\alpha_{i}\delta_{i}\not\in V_{1}^{1}(\sharp I_{1},\varepsilon)=\{u=\sum_{j=1}^{\sharp I_{1}} \alpha_{j}\delta_{j}, a_{j}\in B(y_{l_{j}},\rho), y_{l_{j}}\in
\mathcal K^{+}\; \forall\; j=1,..., \sharp I_{1}\mbox{ and}\; y_{l_{j}}\neq y_{l_{k}} \forall j \neq k\}.$\\
Let $\widetilde{W_{1}}$ be the pseudo-gradient on $V(p,\varepsilon)$ defined by $\widetilde{W_{1}}(u)=W_{1}(u_{1}),$ where $W_{1}$ is the vector filed defined by proposition \ref{3.2} in $ V_{1}(\sharp I_{1},\varepsilon)$. Note that if $u_{1}\not\in V_{1}^{1}(\sharp I_{1},\varepsilon),$ then the pseudo-gradient $W_{1}(u_{1})$ does not increase the maximum of the $\lambda_{i}$'s, $i\in I_{1}$. Using proposition \ref{3.2}, we have

\begin{eqnarray}\label{39}
\Big\langle  \partial J(u), \widetilde{W_{1}}(u) \Big\rangle &\leq& -c \biggr(\dis\sum_{i\in I_{1}} \frac{1}{\lambda_{i}^{\beta_{i}}}
+ \sum_{j \neq i, i,j \in I_{1}}
   \varepsilon_{ij}+\dis\sum_{i\in I_{1}}\frac{|\nabla K(a_{i})|}{\lambda_{i}}\biggl)+ O\biggr(\sum_{i \in I_{1},j \in I_{2} }
   \varepsilon_{ij}\biggr).
\end{eqnarray}
An easy calculation implies that
\begin{eqnarray}\label{a10}
\varepsilon_{ij}&=&o\biggr(\frac{1}{\lambda_{i}^{\beta_{i}}}\biggr)
+o\biggr(\frac{1}{\lambda_{j}^{\beta_{j}}}\biggr),\forall\; i\in I_{1}\; \mbox{and}\;\forall\; j\in I_{2}.
\end{eqnarray}
Fix $i_{0}\in I_{1}$, we denote by
$$J_{1}=\{i\in I_{2},\mbox{ s.t, } \lambda_{i}^{n-2}\geq \frac{1}{2} \lambda_{i_{0}}^{\beta_{i_{0}}} \}\;\mbox{ and}\;J_{2}=I_{2}\setminus J_{1}.$$
Using \eqref{39} and \eqref{a10}, we find that
\begin{eqnarray}\label{41}
\Big\langle  \partial J(u), \widetilde{W_{1}}(u) \Big\rangle &\leq& -c \biggr(\dis\sum_{i\in I_{1}\cup J_{1}} \frac{1}{\lambda_{i}^{\beta_{i}}}
+\sum_{i\in I_{1}}\dis \frac{|\nabla K(a_{i})|}{\lambda_{i}}
+ \sum_{j \neq i \in I_{1}}
   \varepsilon_{ij}\biggl)+ o\biggr( \sum_{i=1}^{p}
  \frac{1}{\lambda_{i}^{\beta_{i}}}\biggr).
\end{eqnarray}
From another part, by Lemma \ref{19} we have
\begin{eqnarray} \label{43}
\Big\langle  \partial J(u), \sum_{i\in J_{1}}-2^{i}Z_{i}(u) \Big\rangle &\leq& c
\sum_{ j \neq i \;,i\in J_{1} } 2^{i}\lambda_{i}\dis \frac{\partial \varepsilon_{ij}}{\partial\lambda_{i}}+O\biggr(\sum_{i\in J_{1}}  \frac{1}{\lambda_{i}^{\beta_{i}}}\biggr)
+ O\biggr( \sum_{i\in J_{1}\cap L_{2}}
 \dis \frac{|(a_{i}-y_{l_{i}})_{k_{i}}|^{\beta_{i}-2}}{\lambda_{i}^{2}}\biggr).
  \end{eqnarray}
Observe that for $i<j$, we have
\begin{equation}\label{epsilona}
2^i\lambda_i\frac{\partial \varepsilon_{ij}}{\partial\lambda_i}+
2^j\lambda_j\frac{\partial \varepsilon_{ij}}{\partial\lambda_j}\leq
-c\; \varepsilon_{ij}.\end{equation}
In addition for $i \in J_{1}$ and $j\in J_{2}$ we have $\lambda_{j}\leq\lambda_{i}$, so by \eqref{pp} we obtain $\lambda_i\dis\frac{\partial \varepsilon_{ij}}{\partial\lambda_i}\leq
-c\; \varepsilon_{ij}.$
These estimates yield
\begin{eqnarray}\nonumber
\Big\langle  \partial J(u), \sum_{i\in J_{1}}-2^{i}Z_{i}(u) \Big\rangle &\leq& -c
\sum_{ j \neq i \;,i\in J_{1},\; j\in J_{1}\cup J_{2} }  \varepsilon_{ij}+O\biggr(\sum_{i\in J_{1}}  \frac{1}{\lambda_{i}^{\beta_{i}}}\biggr)
 \\
 \nonumber&+& O\biggr( \sum_{i\in J_{1}\cap L_{2}}
 \dis \frac{|(a_{i}-y_{l_{i}})_{k_{i}}|^{\beta_{i}-2}}{\lambda_{i}^{2}}\biggr)+O\biggr( \sum_{i\in J_{1},\; j\in I_{1}}\varepsilon_{ij}\biggr).
  \end{eqnarray}
Let $m_{1}>0$ small enough, using Lemma \ref{21} \eqref{ab} and  \eqref{11} we get
\begin{eqnarray}\nonumber
\Big\langle  \partial J(u), \sum_{i\in J_{1}}-2^{i}Z_{i}(u)+m_{1}\sum_{i\in J_{1}\cap L_{2}}X_{i}(u) \Big\rangle &\leq& -c\biggr(
\sum_{ j \neq i \;,i\in J_{1},\; j\in J_{1}\cup J_{2} }  \varepsilon_{ij}+\sum_{i\in J_{1}}\dis \frac{|\nabla K(a_{i})|}{\lambda_{i}}\biggr)\\\nonumber &+&O\biggr(\sum_{i\in J_{1}}  \frac{1}{\lambda_{i}^{\beta_{i}}}\biggr)+ o
\Big(\dis\sum_{i=1}^{p} \frac{1}{\lambda_{i}^{\beta_{i}}}
\Big),
\end{eqnarray}
and by \eqref{41} we obtain

\begin{eqnarray}\nonumber
&&\Big\langle  \partial J(u),\widetilde{W_{1}}(u)+m_{1}\biggr( \sum_{i\in J_{1}}-2^{i}Z_{i}(u)+m_{1}\sum_{i\in J_{1}\cap L_{2}}X_{i}(u)\biggr) \Big\rangle\\
\label{44}&\leq& -c\biggr(\sum_{i\in I_{1}\cup J_{1}}\dis \frac{1}{\lambda_{i}^{\beta_{i}}}
+\sum_{i\neq j \in  I_{1}}\varepsilon_{ij}+
\sum_{ j \neq i \;,i\in J_{1},\; j\in J_{1}\cup J_{2} }  \varepsilon_{ij}\sum_{i\in I_{1}\cup J_{1}}\dis \frac{|\nabla K(a_{i})|}{\lambda_{i}}\biggr)+o
\Big(\dis\sum_{i=1}^{p} \frac{1}{\lambda_{i}^{\beta_{i}}}
\Big).
\end{eqnarray}
We need to add the remainder indices $i\in J_{2}$. Note that $\widetilde{u}:=\dis \sum_{j \in
J_{2}}\alpha_{j}\delta_{j}\in V_{2}(\sharp J_{2},\varepsilon)$. Thus using proposition \ref{3.1}, we apply the associated vector field which  we will
denote $\widetilde{W_{2}}$. We then have the following estimate
\begin{eqnarray}\label{45}
\Big\langle  \partial J(u),\widetilde{W_{2}}(u) \Big\rangle &\leq& -c\biggr(\sum_{j\in J_{2}}\dis \frac{1}{\lambda_{j}^{\beta_{j}}}+\sum_{i\neq j,\; i,j \in  J_{2}}\varepsilon_{ij}+\sum_{j\in  J_{2}}\dis \frac{|\nabla K(a_{j})|}{\lambda_{j}}\biggr)\\\nonumber &+&O\biggr(\sum_{j\in J_{2}, \; i \in J_{1}}\varepsilon_{ij} \biggr)+o
\Big(\dis\sum_{i=1}^{p} \frac{1}{\lambda_{i}^{\beta_{i}}}
\Big),
\end{eqnarray}
since $|a_{i}-a_{j}|\geq\rho$ for $j\in J_{2}$ and $i\in I_{1}$.\\
In this case $W=\widetilde{W_{1}}+m_{1}\biggr( \widetilde{W_{2}}+ \dis\sum_{i\in J_{1}}-2^{i}Z_{i}+m_{1}\dis\sum_{i\in J_{1}\cap L_{2}}X_{i}\biggr).$\\
From \eqref{44} and \eqref{45} we obtain
$$\Big \langle \partial J(u),W(u) \Big \rangle  \leq -c
\biggr( \dis \sum_{i=1}^{p} \dis
\frac{1}{\lambda_{i}^{\beta_{i}}}+ \dis \sum_{i=1}^{p} \dis
\frac{|\nabla K(a_{i}) |}{\lambda_{i}} + \dis \sum_{i \neq j}
\varepsilon_{i j} \biggl).$$

\noindent\underline{case 2.} $u_{1}:=\dis \sum_{i \in
I_{1}}\alpha_{i}\delta_{i}\in  V_{1}^{1}(\sharp I_{1},\varepsilon)$ and $u_{2}:=\dis \sum_{i \not\in
I_{2}}\alpha_{i}\delta_{i}\not\in  V_{2}^{1}(\sharp I_{2},\varepsilon):=\{u=\sum_{j=1}^{\sharp I_{2}}\alpha_{j}\delta_{j}, a_{j}\in B(y_{l_{j}},\rho), y_{l_{j}}\in \mathcal{K}^{+}, \forall j=1,...,\sharp I_{2}\; \mbox{and}\;
\rho(y_{l_{1}},...,y_{\sharp I_{2}}) >0\}.$\\
Since $u_{2}\in V_{2}(\sharp I_{2},\varepsilon)$, by proposition \ref{3.1}, we can apply the associated vector field which we will denote $V_{1}$. We get
\begin{eqnarray}\label{46}
\Big\langle  \partial J(u),V_{1}(u) \Big\rangle &\leq& -c\biggr(\sum_{i\in I_{2}}\dis \frac{1}{\lambda_{i}^{\beta_{i}}}+\sum_{i\in  I_{2}}\dis \frac{|\nabla K(a_{i})|}{\lambda_{i}}+\sum_{i\neq j,\; i,j \in  I_{2}}\varepsilon_{ij}\biggr)+O\biggr(\sum_{i\in I_{2}, \; j \in I_{1}}\varepsilon_{ij} \biggr).
\end{eqnarray}
Observe that $V_{1}(u)$ does not increase the maximum of the $\lambda_{i}$'s, $i\in I_{2}$, since $u_{2}\not\in V_{2}^{1}(\sharp I_{2},\varepsilon)$.
Fix $i_{0}\in I_{2}$ and let
 $$ \widetilde{J_{1}}=\{i\in I_{1},\; s.t,\;\lambda_{i}^{\beta_{i}}\geq \frac{1}{2} \lambda_{i_{0}}^{n-2} \}\; \mbox{and}\; \widetilde{J_{2}}=I_{1}\setminus \widetilde{J_{1}}.$$
Using \eqref{46} and \eqref{a10}, we get
\begin{eqnarray}\label{47}
\Big\langle  \partial J(u),V_{1}(u) \Big\rangle &\leq& -c\biggr(\sum_{i\in I_{2}\cup \widetilde{J_{1}}}\dis \frac{1}{\lambda_{i}^{\beta_{i}}}+\sum_{i\in  I_{2}}\dis \frac{|\nabla K(a_{i})|}{\lambda_{i}}+\sum_{i\neq j,\; i,j \in  I_{2}}\varepsilon_{ij}\biggr)+o
\Big(\dis\sum_{i=1}^{p} \frac{1}{\lambda_{i}^{\beta_{i}}}
\Big).
\end{eqnarray}
We need to add the indices $i$, $i\in  \widetilde{J_{2}}.$ Let $\widetilde u:=\dis \sum_{j \in
\widetilde J_{2}}\alpha_{j}\delta_{j},$ since  $\widetilde {u} \in V_{1}(\sharp \widetilde J_{2},\varepsilon)$, we can apply the associated
vector field giving by proposition \ref{3.1}. Let $V_{2}$ this vector field. By proposition \ref{3.2} we have
\begin{eqnarray}\nonumber
\Big\langle  \partial J(u),V_{2}(u) \Big\rangle &\leq& -c\biggr(\sum_{j\in \widetilde {J_{2}}}\dis \frac{1}{\lambda_{j}^{\beta_{j}}}+\sum_{j\in\widetilde {J_{2}}  }\dis \frac{|\nabla K(a_{j})|}{\lambda_{j}}+\sum_{i\neq j,\; i,j \in  \widetilde {J_{2}}}\varepsilon_{ij}\biggr)+O\biggr(\sum_{j\in \widetilde {J_{2}}, \; i \not\in \widetilde {J_{2}}}\varepsilon_{ij} \biggr).
\end{eqnarray}
Observe that $I_{1}= \widetilde {J_{1}}\cup\widetilde {J_{2}}$ and we are in the case where $\forall \; i \neq j \in I_{1}$, we have $|a_{i}-a_{j}|\geq\rho$. Thus by \eqref{11} and \eqref{a10}, we get
$$ O\biggr(\dis\sum_{j\in \widetilde {J_{2}}, i\not\in  \widetilde {J_{2}} }\varepsilon_{ij}\biggr)=o
\Big(\dis\sum_{i=1}^{p} \frac{1}{\lambda_{i}^{\beta_{i}}}
\Big),$$
 and hence
$$\Big \langle \partial J(u),V_{1}(u)+V_{2}(u) \Big \rangle  \leq -c
\biggr( \dis \sum_{i=1}^{p} \dis
\frac{1}{\lambda_{i}^{\beta_{i}}}+ \dis \sum_{i\in I_{2}\cup \widetilde {J_{2}}} \dis
\frac{|\nabla K(a_{i}) |}{\lambda_{i}} + \dis \sum_{i \neq j}
\varepsilon_{i j} \biggl).$$
Let in this case $W=V_{1}+V_{2}+m_{1}\dis \sum_{i\in\widetilde {J_{1}}}X_{i}(u)$, $m_{1}$ small enough.\\
Using the above estimate and Lemma \ref{21}, we find that
$$\Big \langle \partial J(u),W(u) \Big \rangle  \leq -c
\biggr( \dis \sum_{i=1}^{p} \dis
\frac{1}{\lambda_{i}^{\beta_{i}}}+ \dis \sum_{i=1}^{p} \dis
\frac{|\nabla K(a_{i}) |}{\lambda_{i}} + \dis \sum_{i \neq j}
\varepsilon_{i j} \biggl).$$

\noindent\underline{case 3.} $u_{1}\in V_{1}^{1}(\sharp I_{1},\varepsilon)$ and  $u_{2}\in V_{2}^{1}(\sharp I_{2},\varepsilon).$\\
Let $\widetilde{V_{1}}$ (respectively $\widetilde{V_{2}}$) be the pseudo-gradient in $V(p,\varepsilon)$ defined by $\widetilde{V_{1}}(u)=W_{1}(u_{1})$
(respectively  $\widetilde{V_{2}}(u)=W_{2}(u_{2})$) where $W_{1}$ (respectively $W_{2}$) is the vector field defined by proposition \ref{3.2} (respectively \ref{3.1}) in $V_{1}^{1}(\sharp I_{1},\varepsilon)$ (respectively $V_{2}^{1}(\sharp I_{2},\varepsilon)$) and let in this case $ W=  \widetilde{V_{1}}+\widetilde{V_{2}}.$\\
Using proposition \ref{3.1}, proposition \ref{3.2} and \eqref{a10} we get
$$\Big \langle \partial J(u),W(u) \Big \rangle  \leq -c
\biggr( \dis \sum_{i=1}^{p} \dis
\frac{1}{\lambda_{i}^{\beta_{i}}}+ \dis \sum_{i=1}^{p} \dis
\frac{|\nabla K(a_{i}) |}{\lambda_{i}} + \dis \sum_{i \neq j}
\varepsilon_{i j} \biggl).$$
Notice that in the first and second cases, the maximum of the $\lambda_{i}$'s, $1\leq i\leq p$ is a bounded function and hence the Palais-Smail condition is satisfied along the flow-lines of $W$. However in the third case all the $\lambda_{i}$'s, $1\leq i\leq p$, will increase and goes to $+\infty$ along the flow-lines generated by $W$.\\

\noindent\underline{Subset 2}. We consider the
case of $u=\dis \sum_{i=1}^{p} \alpha_{i} \delta_{i} \in V(p,
\varepsilon) $, such that there exist $a_{i}$ satisfying
$a_{i}\notin \dis\cup_{y\in \mathcal{K} }B(y, \rho)$. We order the $\lambda_i$s in an increasing order, without loss of generality, we
suppose that $\lambda_1\leq \cdots \leq\lambda_p$. Let $i_1$ be such that for any $i < i_1$, we have $a_i\in B(y_{\ell_i}, \rho), y_{\ell_i}\in \mathcal{K}$ and $a_{i_1} \notin \dis\cup_{y\in \mathcal{K} }B(y, \rho)$. Let us define $$u_1= \sum_{i<i_1}\alpha_i\delta_i.$$

Observe that $u_1$ has to satisfy one of three cases above that is $u\in V_1(i_1-1, \varepsilon)$ or $u_1\in V_2(i_1-
1, \varepsilon)$ or $u_1$ satisfies the condition of subset 1. Thus we can apply the associated vector field which
we will denote by $Y$ and we have then the following estimate.

$$\Big \langle \partial J(u), Y(u) \Big \rangle  \leq -c
\biggr( \dis \sum_{i<i_1} \dis
\frac{1}{\lambda_{i}^{\beta_{i}}}+ \dis \sum_{i<i_1} \dis
\frac{|\nabla K(a_{i}) |}{\lambda_{i}} + \dis \sum_{i \neq j, ij<i_1}
\varepsilon_{i j} \biggl)+ O\biggl(\dis \sum_{i<i_1, j\geq i_1}
\varepsilon_{i j} \biggl).$$
Now we define the following vector field

$$Y'= \frac{1}{\lambda_{i_1}}\frac{\partial \delta_{i_1}}{\partial a_{i_1}}\frac{\nabla K(a_{i_1})}{|\nabla K(a_{i_1})|}-c'\sum_{i\geq i_1}2^i Z_i.$$
Using Propositions \ref{3.1}, \ref{3.2} and the fact that $|\nabla K(a_{i_1})|\geq c >0$, we derive

$$\Big \langle \partial J(u), Y'(u) \Big \rangle  \leq -c \frac{1}{\lambda_{i_1}} + O\biggl(\dis \sum_{i \neq i_1}
\varepsilon_{i j} \biggl) -c'  \sum_{j\neq i, i\geq i_1}
\varepsilon_{i j}  + o\biggl( \sum_{i\geq i_1}\frac{1}{\lambda_{i}}\biggl).$$
Taking $c'$ positive large enough, we find

$$\Big \langle \partial J(u), Y'(u) \Big \rangle  \leq -c
\biggr( \dis \sum_{i=i_1}^p \dis
\frac{1}{\lambda_{i}^{\beta_{i}}}+ \dis \sum_{i=i_1}^p\dis
\frac{|\nabla K(a_{i}) |}{\lambda_{i}} + \dis \sum_{i \neq j, i\geq i_1}
\varepsilon_{i j} \biggl).$$
Now, let $W := Y' +m_1 Y$ where $m_1$ is a small positive constant, then we have

$$\Big \langle \partial J(u),W(u) \Big \rangle  \leq -c
\biggr( \dis \sum_{i=1}^{p} \dis
\frac{1}{\lambda_{i}^{\beta_{i}}}+ \dis \sum_{i=1}^{p} \dis
\frac{|\nabla K(a_{i}) |}{\lambda_{i}} + \dis \sum_{i \neq j}
\varepsilon_{i j} \biggl).$$

\noindent Finally, observe
that our pseudo-gradient $W$ in $V(p,\varepsilon)$ satisfies claim $(i)$ of Theorem \ref{1} and it is bounded, since
 $||\dis \lambda_{i}\frac{\partial\delta_{(a_{i,\lambda_{i}})}}{\partial\lambda_{i}}||$ and
 $||\dis \frac{1}{\lambda_{i}}\frac{\partial\delta_{(a_{i,\lambda_{i}})}}{\partial a_{i}}||$ are bounded. From the definition of $W$, the $\lambda_{i}$'s,
 $1\leq i\leq p$ decrease along the flow-lines of $W$ as long as these flow-lines do not enter in the neighborhood of finite number of critical points $y_{l_{i}}$,
 $i=1,...,p$, of  $\mathcal{K}$ such that $(y_{l_{1}},...,y_{l_{p}})\in \mathcal P^{\infty}$.\\
 Now, arguing as in Appendix 2 of \cite{b2}, claim $(ii)$ of Theorem \ref{3.1} follows from $(i)$ and proposition \ref{p24}. This complete the proof of
 Theorem \ref{1}.

\vskip0.2cm
\noindent\textbf{Proof of  Proposition \ref{3.2}.}
In our construction of the pseudo gradient $W_{1}$,  we need  the following lemmas:

\begin{lem}\label{19}
Let $u=\dis \sum_{i=1}^{p}
\alpha_{i} \delta_{i}\in
V(p,\varepsilon)$, such that  $a_{i}\in B(y_{l_{i}},\rho)$, $y_{l_{i}}\in\mathcal{K}$, $\forall\; i=1,...,p$. We then have
\begin{eqnarray}\nonumber
 <\partial J(u), Z_i(u)>&=&-2c_{2} J(u) \sum_{j\neq i}\alpha_i \alpha_{j}\frac{\partial\varepsilon_{ij}}{\partial \lambda_i}+\dis O\biggr(\dis \frac{1}{\lambda_{i}^{\beta_{i}}}\biggr)\\
 \nonumber&+&\biggr[ O\biggr(\frac{|(a_{_{i}}-y_{l_{i}})_{_{k_{i}}}|^{\beta_{i}-2}}{\lambda_{i}^{2}} \biggr), \mbox{if}\; i\in L_{2}\biggr ]
+ o\Big(\sum_{j\neq i} \varepsilon_{ij}\Big)+o\biggr(\sum_{j=1}^{p}
  \dis\frac{1}{\lambda_{j}^{\beta_{j}}}\biggr),
\end{eqnarray}
where $k_{i}$ is defined in \eqref{M}.
\end{lem}

\begin{proof}
Observe that for $k\in\{ 1,...,n\}$, if $\lambda_{i}|(a_{i}-y_{l_{i}})_{k}|>\frac{M_{1}}{\sqrt{n}}$, we have
\begin{eqnarray}\label{20}\label{47}
\dis \int_{\dis \mathbb{R}^{n}} \dis
\frac{\big|x_{k} + \lambda_{i} (a_{i}-y_{l_{i}})_{k} \big|^{\beta_{i}-1} x_{k}}{(1+ |
x|^{2})^{n}}dx
= O\Big(
(\lambda_{i} |(a_{i}-y_{l_{i}})_{k} |)^{\beta_{i}-2}\Big),
\end{eqnarray}
taking $M_{1}$ large enough. If not, we have
\begin{eqnarray}\nonumber
\dis \int_{\dis \mathbb{R}^{n}} \dis
\frac{\big|x_{k} + \lambda_{i} (a_{i}-y_{l_{i}})_{k} \big|^{\beta_{i}-1} |x_{k}|}{(1+ |
x|^{2})^{n}}dx
&=& O(1).
\end{eqnarray}
Using the fact that $k_{i}$ defined in  \eqref{M} satisfies $\lambda_{i}|(a_{i}-y_{l_{i}})_{k_{i}}|>\frac{M_{1}}{\sqrt{n}}$, if $i\in L_{2}$, Lemma
\ref{19} follows from Proposition \ref{p3.2}
\end{proof}

\begin{lem}\label{21}
For $u=\dis \sum_{i=1}^{p}
\alpha_{i} \delta_{i}\in
V(p,\varepsilon)$, such that  $a_{i}\in B(y_{l_{i}},\rho)$, $y_{l_{i}}\in\mathcal{K}$, $\forall\; i=1,...,p$, we have
\begin{eqnarray}\nonumber
 <\partial J(u), X_i(u)>&\leq&O\biggr( \sum_{j\neq i}\frac{1}{\lambda_{i}}|\frac{\partial\varepsilon_{ij}}{\partial a_i}|\biggr)+\dis O\biggr[\biggr(\dis \frac{1}{\lambda_{i}^{\beta_{i}}}\biggr),\mbox{if} \; i\in L_{1}\biggr]\\
 \nonumber&+&\biggr[-c \biggr(\frac{1}{\lambda_{i}^{\beta_{i}}} +\frac{|(a_{_{i}}-y_{l_{i}})_{_{k_{i}}}|^{\beta_{i}-1}}{\lambda_{i}} \biggr), \mbox{if}\; i\in L_{2}\biggr ]
+o\biggr(\sum_{j=1}^{p}
  \dis\frac{1}{\lambda_{j}^{\beta_{j}}}\biggr),
\end{eqnarray}
where $k_{i}$ is defined in \eqref{M}.
\end{lem}

\begin{proof}
Using proposition \ref{p3.3}, we have
\begin{eqnarray}\nonumber
 <\partial J(u), X_i(u)>
&\leq& -c \frac{1
}{{\lambda_i}^{\beta_{i}}}
\Bigg( \int_{\dis \mathbb{R}^{n}} b_{k_{i}}
\frac{|x_k+\lambda_i(a_i-y_{l_{_{i}}})_{k_{i}}|^{\beta_{i}}}
{(1+\lambda_i|(a_i-y_{l_{i}})_{k_{i}}|)^{(\beta_{i}-1)/2}}
\frac{x_{k_{i}}}{(1+|x|^2)^{n+1}}dx\Bigg)^2 \\
\label{23}
&+& O\biggr( \sum_{j\neq i}\frac{1}{\lambda_{i}}|\frac{\partial\varepsilon_{ij}}{\partial a_i}|\biggr)+o\biggr(\sum_{j=1}^{p}
  \dis\frac{1}{\lambda_{j}^{\beta_{j}}}\biggr).
\end{eqnarray}
Using \eqref{eq4.1} and the fact that $\lambda_{i}|(a_{i}-y_{l_{i}})_{k_{i}}|>\frac{M_{1}}{\sqrt{n}}$, if $i\in L_{2}$, lemma \ref{21} follows.
\end{proof}

\vskip0.2cm

In order to construct the required pseudo-gradient,
we  have to divide the set  $V_{1}(p,\varepsilon)$ in four different regions, to
construct an appropriate pseudo-gradient in each region and then glue up through convex combinations. Let\\
$V^{1}_{1}(p,\varepsilon) := \Big\{u=\dis \sum_{i=1}^{p}\alpha_{i}
\delta_{(a_{i} \lambda_{i})}\in V_{1}(p,\varepsilon),\;
y_{l_{i}}\neq y_{l_{j}} ,\; \forall i \neq j,\;-
\sum_{k=1}^{n}b_{k}(y_{l_{i}}) > 0, \mbox{and
}\lambda_{i}|a_{i}-y_{l_{i}}|< \delta,\; \forall i=1,...,p \Big\}.$\\
$V^{2}_{1}(p,\varepsilon) := \Big\{u=\dis \sum_{i=1}^{p}\alpha_{i}
\delta_{(a_{i} \lambda_{i})}\in V_{1}(p,\varepsilon),\;
y_{l_{i}}\neq y_{l_{j}}, \; \forall i \neq
j,\;\lambda_{i}|a_{i}-y_{l_{i}}|< \delta,\; \forall i=1,..,p
\mbox{ and there exist }  \exists\; i_{1},...,i_{q}\\ \mbox{ such
that }-  \sum_{k=1}^{n}b_{k}(y_{l_{i_{j}}}) < 0, \;\forall
j=1,...,q \Big\}.$\\
$
V^{3}_{1}(p,\varepsilon) := \Big\{u=\dis \sum_{i=1}^{p}\alpha_{i}
\delta_{(a_{i} \lambda_{i})}\in V_{1}(p,\varepsilon),\;
y_{l_{i}}\neq y_{l_{j}}, \; \forall i \neq j,\; \mbox{and there
exist } j \mbox{ (at least)}, s, t, \lambda_{j}|a_{j}-y_{l_{j}}|\geq\dis \frac{\delta}{2}  \Big\}.
$\\
$
V^{4}_{1}(p,\varepsilon) := \Big\{u=\dis \sum_{i=1}^{p}\alpha_{i}
\delta_{(a_{i} \lambda_{i})}\in V_{1}(p,\varepsilon),\;
 \mbox{such that there exist } i\neq j
  \mbox{ satisfying }y_{l_{i}}= y_{l_{j}} \Big\}.
$\\

\noindent\underline {Pseudo-gradient in $V_{1}^{1}(p,\varepsilon)$}. Let  $u=\dis \sum_{i=1}^{p}
\alpha_{i} \delta_{i}\in
V_{1}^{1}(p,\varepsilon)$. For any $i\neq j$, we have $| a_{i}-a_{j}|> \rho$, therefore

\begin{equation}\label{11}
\varepsilon_{ij}= O\biggr(\dis\frac{1}{(\lambda_{i}\lambda_{j})^{\frac{n-2\sigma}{2}}}\biggr)=o\biggr(\frac{1}{\lambda_{i}^{\beta_{i}}}\biggr)
+o\biggr(\frac{1}{\lambda_{j}^{\beta_{j}}}\biggr),
\end{equation}
since $\beta_{i}, \beta_{j}<n-2\sigma$. Let $W_{1}^{1}(u)=\dis \sum_{i=1}^{p} Z_{i}(u)$, using the fact that $\dis\frac{|\nabla
K(a_i)|}{\lambda_i}$ is small with respect to $\dis\frac{1}{{\lambda_i}^{\beta}}$, we obtain from Proposition \ref{p3.2}
$$ \Big \langle \partial J(u), W^{1}_{1}(u)  \Big \rangle \leq -c \biggr( \dis \sum_{i=1}^{p} \dis \frac{1}
{\lambda_{i}^{\beta_{i}}}+ \dis \sum_{i=1}^{p} \dis \frac{|\nabla
K(a_{i}) |}{\lambda_{i}} + \dis \sum_{i \neq j} \varepsilon_{i j}
\biggl). $$
\underline {Pseudo-gradient in $V_{1}^{2}(p,\varepsilon)$}. Let $u =\dis \sum_{i=1}^{p}
\alpha_{i} \delta_{i}\in
V_{1}^{2}(p,\varepsilon)$. Without loss of generality, we can assume that $1,...,q$ are the indices which satisfy $-\sum_{k=1}^{n}b_{k}(y_{l_{i}})<0$,
$\forall i=1,...,q$. Let $$I= \Big \{ i=1,...,p \;\mbox{ s.t }\;
\lambda_{i}^{\beta_{i}}\leq \dis \frac{1}{10} \dis \min_{1 \leq j
\leq q} \lambda_{j}^{\beta_{j}} \Big \}.$$
 In this region we define $W_{1}^{2}(u)= \dis \sum_{i=1}^{q} (- Z_{i})(u) +\dis \sum_{i\in I}  Z_{i}(u).$ Using similar calculation than \cite{BC3}, we obtain
$$ \Big \langle \partial J(u), W^{2}_{1}(u)  \Big \rangle \leq -c \biggr( \dis \sum_{i=1}^{p} \dis \frac{1}
{\lambda_{i}^{\beta_{i}}}+ \dis \sum_{i=1}^{p} \dis \frac{|\nabla
K(a_{i}) |}{\lambda_{i}} + \dis \sum_{i \neq j} \varepsilon_{i j}
\biggl). $$
\underline {Pseudo-gradient in $V_{1}^{3}(p,\varepsilon)$}. Let $u =\dis \sum_{i=1}^{p}
\alpha_{i} \delta_{i}\in
V_{1}^{3}(p,\varepsilon)$. Without loss of generality, we can assume that $\lambda_{1}^{\beta_{1}}= \min\{ \lambda_{j}^{\beta_{j}}
\mbox{ s.t }
\lambda_{j}|a_{j}-y_{l_{j}} | \geq \delta\}.$ Let $J:=\Big \{ i, 1 \leq i \leq p \;\mbox{ s.t }\;
\lambda_{i}^{\beta_{i}} \geq \dis \frac{1}{2} \dis
\lambda_{1}^{\beta_{1}} \Big \}.$ Observe that if $i \notin J$ we have $\lambda_{i}|a_{i}-y_{l_{i}} | \geq \delta $. We write $u$ as follows $ u= \dis \sum_{i \in J^{C}}\alpha_{i} \delta_{i}+ \dis \sum_{i \in J}\alpha_{i} \delta_{i}= u_{1}+ u_{2}. $ Observe that $u_{1}$ has to satisfy one of two above cases that is $u_{1}\in V_{1}^{1}(\sharp J^{C} , \varepsilon)$ or
$u_{1}\in V_{1}^{2}(\sharp J^{C}, \varepsilon)$. Let $\widetilde{W}$ be a pseudo-gradient on $V_{1}^{3}(p, \varepsilon)$
defined by $\widetilde{W}(u)=W_{1}^{1}(u_{1})$, if $u_{1}\in V_{1}^{1}(\sharp J^{C} , \varepsilon)$ or $\widetilde{W}(u)=W_{1}^{2}(u_{1})$, if
$u_{1}\in V_{1}^{2}(\sharp J^{C}, \varepsilon)$. Let in this region $W_1^{3}(u)=\widetilde{W}(u)+X_1(u)+ \dis \sum_{i \in J \cap L_{2}}X_i(u)- M_1 Z_1(u)$. By Propositions \ref{p3.2} and \ref{p3.3}, we have
$$ \Big \langle \partial J(u), W_{1}^{3}(u) \Big \rangle \leq -c \biggr( \dis \sum_{i=1}^{p}
 \dis \frac{1}{\lambda_{i}^{\beta_{i}}}
+ \dis \sum_{i=1}^{p} \dis \frac{|\nabla K(a_{i}) |}{\lambda_{i}}
+ \dis \sum_{i \neq j} \varepsilon_{i j} \biggl) .$$
\underline {Pseudo-gradient in $V_{1}^{4}(p,\varepsilon)$}. We study now the case of u=$\dis \sum_{i=1}^{p}
\alpha_{i} \delta_{i}\in
V_{1}^{4}(p,\varepsilon)$.
Let, $B_{k}=\{j,\; 1\leq j\leq p\; \mbox{ s.t }\; a_{j} \in
B(y_{l_{k}},\rho) \}.$ In this case, there is at least one
$B_{k}$ which contains at least two indices. Without loss of
generality, we can assume that $1,...,q$ are the indices such that
the set $B_{k}$, $1\leq k \leq q$ contains at least two indices.
We will decrease the $\lambda_{i}$'s for $i \in B_{k}$ with
different speed. For this purpose, let $$\begin{array}{ccc}
  \chi : \mathbb{R}& \longrightarrow & \mathbb{R}^{+} \\
   t& \longmapsto &  \left\{%
\begin{array}{ll}
    0 & \hbox{  if    }| t|\;\leq \widetilde{\gamma}  \\
    1 & \hbox{  if    } | t|\;\geq 1.
\end{array}%
\right.
\end{array}$$
Here $\widetilde{\gamma}$ is a small constant.
For $j\in B_{k}$, set $\overline{\chi}\;(\lambda_{j})= \dis
\sum_{i \neq j,\; i \in B_{k}} \chi \; \Big(\dis
\frac{\lambda_{j}}{\lambda_{i}}\Big) $. Let, $I_{1}=\Big\{ i,\; 1\leq i \leq p, \; \lambda_{i}|a_{i}-y_{l_{i}}
|\geq \delta \Big\}.$\\
We distinguish two cases:\\
\underline{case 1.}  $I_{1}\neq \emptyset$, let in this case  $$J=\Big\{ j, 1\leq j\leq p,s.t,\; \lambda_{j}^{\beta_{j}}\geq \dis\frac{1}{2}\dis \min_{i
\in I_{1}} \lambda_{i}^{\beta_{i}} \Big\}.$$
Observe that, if $a_i \in B(\;y_{l_i},\,\rho\,)$, we have $
|\nabla K(a_{i}) | \sim \sum_{k=1}^{n}|b_k|
|(a_{i}-y_{l_{i}})_{k} |^{\beta_{i}-1}$. So, if $i\in L_1$ we have  $\dis
\frac{|\nabla K(a_{i}) |}{\lambda_{i}}\leq
\dis \frac{c}{\lambda_{i}^{\beta_{i}}}$, and if $i\in L_2$ we have
$\dis
\frac{|\nabla K(a_{i}) |}{\lambda_{i}} \leq\, c\, \dis
\frac{|(a_{i}-y_{l_{i}})_{k} |^{\beta_{i}-1}}{ \lambda_{i}}$. Thus by lemma \ref{21} we obtain
\begin{eqnarray*}
  \Big \langle \partial
J(u),\sum_{i \in I_{1}}X_{i}(u) \Big \rangle & \leq &-c_{\delta}
\biggr(\sum_{i \in J} \dis \frac{1}{\lambda_{i}^{\beta_{i}}}+  \sum_{i
\in J }\dis \frac{|\nabla K(a_{i}) |}{\lambda_{i}}+\sum_{i\in I_{1}\cap L_{2}}\dis \frac{|(a_{i}-y_{l_{i}})|^{\beta_{i}-1}}{\lambda_{i}}\biggr) \\ &+ &O \Big(
\sum_{i \neq j,\; i \in I_{1}} \Big | \dis \frac{1}{\lambda_{i}}
\dis \frac{\partial \varepsilon_{i j}}{\partial a_{i}}\;
\Big| \Big)+o \biggr(  \dis
\sum_{i=1}^{p}  \dis \frac{1}{\lambda_{i}^{\beta_{i}}}\biggl).
\end{eqnarray*}
Let $\widetilde{C}=\Big\{(i,j)\mbox{ s.t } \gamma \leq
\frac{\lambda_{i}}{\lambda_{j}} \leq \frac{1}{\gamma} \Big\},$  where $\gamma$ is a small positive constant. Observe that
$$  \Big | \dis
\frac{1}{\lambda_{i}} \dis \frac{\partial \varepsilon_{i
j}}{\partial a_{i}}\; \Big|= o (\varepsilon_{ij}), \,\forall\; i\neq j\,\in \widetilde{C}. $$
This with  \eqref{eq4.1} yields
\begin{eqnarray}
\label{33} \Big \langle \partial J(u),\sum_{i \in I_{1}}X_{i}(u)
\Big \rangle &\leq& -c_{\delta}
\biggr(\sum_{i \in J} \dis \frac{1}{\lambda_{i}^{\beta_{i}}}+  \sum_{i
\in J }\dis \frac{|\nabla K(a_{i}) |}{\lambda_{i}}+\sum_{i\in I_{1}\cap L_{2}}\dis \frac{|(a_{i}-y_{l_{i}})|^{\beta_{i}-1}}{\lambda_{i}}\biggr)
 \\\nonumber  &+& o \Big(
\dis\sum_{k=1}^{q}\sum_{i \neq j\in B_{k}\cap \widetilde{C},\; i
\in I_{1}} \varepsilon_{i j} \Big)
 + O  \Big( \dis\sum_{k=1}^{q}\sum_{i \neq j\in
B_{k}, (i,j)\not\in \widetilde{C},\; i \in I_{1}} \varepsilon_{i j}
\Big)+o \Big(\dis\sum_{i=1}^{p}
\frac{1}{\lambda_{i}^{\beta_{i}}} \Big).
\end{eqnarray}
For any  $k=1,...,q$, let  $\lambda_{i_{k}}= \min\{\lambda_{i}, i
\in B_{k} \}$. Define
$$\overline {Z}= -\sum_{k=1}^{q}\sum_{ j\in B_{k}, (i_{k}, j) \notin\widetilde{C}} \overline{\chi}(\lambda_{j})
Z_{j}-\gamma_{1} \sum_{k=1}^{q}\sum_{ j\in B_{k}, (i_{k}, j)
\in\widetilde{C}}
\overline{\chi}(\lambda_{j}) Z_{j},$$ where  $\gamma_{1}$ is a  small positive constant. Using Lemma \ref{19}, we find that
\begin{eqnarray}
\nonumber \Big \langle \partial J(u),\overline {Z}(u)
\Big \rangle &\leq& c\dis\sum_{k=1}^{q}\sum_{i \neq j, j\in B_{k}, (j,i_{k})\not\in \widetilde{C}}\overline{\chi}\;(\lambda_{j})\lambda_{j} \dis\frac{\partial\varepsilon_{i j}}{\partial\lambda_{j}}
 +c\gamma_{1}\dis\sum_{k=1}^{q}\sum_{ j\in B_{k}, (j,i_{k})\in \widetilde{C}, i \neq j}\overline{\chi}\;(\lambda_{j})\\
 \nonumber
&\times& \lambda_{j} \dis\frac{\partial\varepsilon_{i j}}{\partial\lambda_{j}}+O\biggr( \dis\sum_{k=1}^{q}\sum_{ j\in B_{k}, (j,i_{k})\not\in \widetilde{C}}
\biggr(\dis\frac{1}{\lambda_{j}^{\beta_{j}}} +\dis\frac{|(a_{j}-y_{l_{j}})|^{\beta_{j}-2}}{\lambda_{j}^{2}},\mbox{if}\; j\in L_{2} \biggr)\biggr)\\
\nonumber
&+& \gamma_{1}O \biggr( \dis\sum_{k=1}^{q}\sum_{ j\in B_{k}, (j,i_{k})\in \widetilde{C}}
\biggr(\dis\frac{1}{\lambda_{j}^{\beta_{j}}} +\dis\frac{|(a_{j}-y_{l_{j}})|^{\beta_{j}-2}}{\lambda_{j}^{2}},\mbox{if}\; j\in L_{2} \biggr)\biggr).
\end{eqnarray}
Observe that by using a direct calculation, we have
\begin{eqnarray}\label{pp}
\dis\lambda_{i}\frac{\partial\varepsilon_{ij}}{\partial\lambda_{i}}\leq-c\; \varepsilon_{ij},\; \mbox{if} \; \dis\lambda_{i}\geq\lambda_{j}
\;\mbox{or}\;  \dis\lambda_{i}\sim\lambda_{j}\; \mbox{or}\; |a_{i}-a_{j}|\geq\delta_{0}>0.
\end{eqnarray}
Let $j\in B_{k}$, $1\leq k\leq q$ and let $i$, $1\leq i\leq p$ such that $i \neq j$. If  $i\not\in B_{k}$ or  $i\in B_{k}$, with $(i, j)\in \widetilde{C}$,
then we have by \eqref{pp}
$$ \lambda_{i}\frac{\partial\varepsilon_{i j}}{\partial \lambda_{i}} \leq -c\; \varepsilon_{ij}\mbox{ and}\; \lambda_{j}\frac{\partial\varepsilon_{i j}}{\partial \lambda_{j}} \leq -c \;\varepsilon_{ij}.$$
In the case where $i\in B_{k}$ with $(i, j)\not\in \widetilde{C}$, (assuming that $\lambda_{i}<<\lambda_{j}$), we have $\overline{\chi}(\lambda_{j}) - \overline{\chi}(\lambda_{i}) \geq 1
    $. Thus,
$$ \overline{\chi}\;(\lambda_{j})\;\lambda_{j} \dis \frac{\partial
    \varepsilon_{i j}}{\partial\lambda_{j}}+ \overline{\chi}\;(\lambda_{i})\;\lambda_{i} \dis \frac{\partial
    \varepsilon_{i j}}{\partial\lambda_{i}} \leq \lambda_{j} \dis \frac{\partial
    \varepsilon_{i j}}{\partial\lambda_{j}} \leq -c\; \varepsilon_{i j}. $$
We therefore have
\begin{eqnarray}
\nonumber \Big \langle \partial J(u),\overline {Z}(u)
\Big \rangle &\leq& -c\biggr(\dis\sum_{k=1}^{q}\sum_{i \neq j, j\in B_{k}, (j,i_{k})\not\in \widetilde{C}}\varepsilon_{i j}
 +\gamma_{1}\dis\sum_{k=1}^{q}\sum_{i\neq j, j\in B_{k}, (j,i_{k})\in \widetilde{C}}\varepsilon_{i j}\biggr)\\
 \nonumber&+&O\biggr( \dis\sum_{k=1}^{q}\sum_{ j\in B_{k}, (j,i_{k})\not\in \widetilde{C}}
\biggr(\dis\frac{1}{\lambda_{j}^{\beta_{j}}} +\dis\frac{|(a_{j}-y_{l_{j}})|^{\beta_{j}-2}}{\lambda_{j}^{2}},\mbox{if}\; j\in L_{2} \biggr)\biggr)\\
\label{34}
&+& \gamma_{1}O \biggr( \dis\sum_{k=1}^{q}\sum_{ j\in B_{k}, (j,i_{k})\in \widetilde{C}}
\biggr(\dis\frac{1}{\lambda_{j}^{\beta_{j}}} +\dis\frac{|(a_{j}-y_{l_{j}})|^{\beta_{j}-2}}{\lambda_{j}^{2}},\mbox{if}\; j\in L_{2} \biggr)\biggr).
\end{eqnarray}
Observe that if $j \in B_{k}$ with $(j,i_{k})\in \widetilde{C}$, we have $j$ or $i_{k}\in I_{1}$. Thus for $M_{1}$ large enough, and $\gamma_{1}$ very small, we
obtain from  \eqref{33} and \eqref{34}
\begin{eqnarray}
\nonumber \Big \langle \partial J(u), \sum_{i \in I_{1}}X_{i}+M_{1}\overline {Z}(u) \Big \rangle  &\leq&
-c\biggr(\sum_{i \in J} \dis
\frac{1}{\lambda_{i}^{\beta_{i}}}+ \sum_{i \in J}\dis
\frac{|\nabla K(a_{i}) |}{\lambda_{i}} +\sum_{k=1}^{q}\;\;
\sum_{i\neq j, j\in B_{k}}\varepsilon_{ij}\biggr)  \\
\label{a4} &+&O\Big(
\sum_{k=1}^{q}\;\;\sum_{ j\in B_{k}, (i_{k}, j)
\notin\widetilde{C}} \frac{1}{\lambda_{j}^{\beta_{j}}} \Big),
\end{eqnarray}
\begin{equation}\label{ab}
\mbox{ since  } \hskip1.5cm\dis
\frac{|(a_{i}-y_{l_{i}})_{k_i} |^{\beta_{i}-2}}{ \lambda_{i}^{2}}=o \Big( \dis
\frac{|(a_{i}-y_{l_{i}})_{k_i} |^{\beta_{i}-1}}{ \lambda_{i}} \Big),\, \mbox{for any } i\in L_2,\, \mbox{as } M_1 \mbox{ large enough}.\hskip2.3cm
\end{equation}
Now, let in this region
$$ W_{1}^{4}:=M_{1}\biggr( \sum_{i \in I_{1}}X_{i}+ M_{1}\overline {Z}\biggr) +\dis \sum_{i \notin J} (-\sum_{k=1}^{n} b_{k}) Z_{i}. $$
We obtain from the above estimates
$$\Big \langle \partial J(u),W_{1}^{4}(u) \Big \rangle  \leq -c
\biggr( \dis \sum_{i=1}^{p} \dis \frac{1}{\lambda_{i}^{\beta_{i}}}+
\dis \sum_{i=1}^{p} \dis \frac{|\nabla K(a_{i}) |}{\lambda_{i}} +
\dis \sum_{i \neq j} \varepsilon_{i j} \biggl).$$

\noindent\underline{case 2.} $I_{1} = \emptyset$, we order the $\lambda_{i}$'s in an increasing order, for sake of simplicity, we can assume that $\lambda_{1}\leq...\leq \lambda_{p}$.
Let  \\
$$I_{2}=\{ 1 \} \cup \{ i, 1\leq i\leq p \; s.t\; \lambda_{i}\sim
\lambda_{1} \}.$$
We write $u$ as follows $$u= \dis \sum_{i \in
I_{2}}\alpha_{i}\delta_{i}+ \dis \sum_{i \notin
I_{2}}\alpha_{i}\delta_{i}:=u_{1}+u_{2}.$$
Observe that, $\forall i \neq j \in I_{2}$ such that $i\neq j$ we
have $|a_{i}-a_{j}|\geq \delta$. Indeed, if $|a_{i}-a_{j}|<\delta$, so $i,j\in B_{k}$, we get $|a_{i}-a_{j}|\leq|a_{i}-y_{l_{i}}|+|a_{j}-y_{l_{i}}|\leq\dis\frac{2\delta}{\lambda_{i}},$
since $I_{1}=\emptyset$ and $\lambda_{i}\sim\lambda_{j}$ $\forall\; i,j \in I_{2}$. This implies that
$$\biggr(\displaystyle
\frac{\displaystyle\lambda_{i}}{\displaystyle\lambda_{j}}+\displaystyle
\frac{\displaystyle\lambda_{j}}{\displaystyle\lambda_{i}}+\displaystyle
\displaystyle\lambda_{i}\displaystyle\lambda_{j}|a_{i}-a_{j}|^{2}\biggr)^{\frac{n-2\sigma}{2}}\leq c_{1},$$
and hence $\varepsilon_{ij}\geq c$ which is a contradiction. Thus  $u_{1}\in
V_{1}^{j}(\sharp I_{2}, \varepsilon)$, $j=1$ or $2$ or $3$. Apply the associated pseudo-gradient denoted by $\overline W$, we obtain
$$\Big \langle \partial J(u),\overline W(u) \Big \rangle  \leq -c \biggl(\dis \sum_{i \in I_{2}}\dis
\frac{1}{\lambda_{i}^{\beta_{i}}}+  \dis \sum_{i\in  I_{2}} \dis
\frac{|\nabla K(a_{i}) |}{\lambda_{i}} + \dis \sum_{i \neq j,\;
i,j\in I_{2} } \varepsilon_{i j} \biggr)+O \Big(\sum_{i \in
I_{2},j\notin I_{2} } \varepsilon_{i j} \Big).$$
$$ \mbox{Let,} \hskip3cm J_{2} =\{i, 1\leq i\leq p , \; \lambda_{i}^{\beta_{i}}\geq  \dis \min_{j\in I_{2}} \lambda_{j}^{\beta_{j}}  \}.\hskip8.3cm $$
We can add to the above estimates all indices $i$ such that
$ i \in J_{2}$. So using the estimate \eqref{11} we obtain
$$
\nonumber \Big \langle \partial J(u),\overline{W}(u) \Big
\rangle   -c\biggl(\dis \sum_{i \in J_{2}}\dis
\frac{1}{\lambda_{i}^{\beta_{i}}}+  \dis \sum_{i\in  J_{2}} \dis
\frac{|\nabla K(a_{i}) |}{\lambda_{i}} +\dis \sum_{i \neq j,\;
i,j\in I_{2} } \varepsilon_{i j}\biggr)$$$$+ o \Big(\dis \sum_{i=1}^{p}\dis
\frac{1}{\lambda_{i}^{\beta_{i}}} \Big) +O \Big(\sum_{i \in
I_{2},j\notin I_{2}, i,j \in B_{k} } \varepsilon_{i j} \Big).
$$
Let $M_{1}>0$ large enough, the above estimate and \eqref{34} yields
\begin{eqnarray}
\nonumber \Big \langle \partial J(u),M_{1}\overline{Z}(u) +\overline{W}(u) \Big
\rangle &\leq& -c\biggr(\dis \sum_{i \in J_{2}}\dis
\frac{1}{\lambda_{i}^{\beta_{i}}}+ \dis \sum_{i\in J_{2}} \dis
\frac{|\nabla K(a_{i}) |}{\lambda_{i}} + \dis
\sum_{k=1}^{q}\;\; \sum_{i \neq j \in B_{k} } \varepsilon_{i j}\\
&+& \dis \sum_{i \neq j,\; i,j\in I_{2} }
\varepsilon_{i j}\biggr)  +O \Big(\dis
\sum_{k=1}^{q}\;\; \sum_{i\in B_{k},(i_{k},i)\notin \widetilde{C}
}\frac{1}{\lambda_{i}^{\beta_{i}}}\Big).
\end{eqnarray}
From another part, by $(iii)$ of proposition \ref{3.1} and \eqref{11}, we have
\begin{eqnarray}\label{38}
\Big \langle \partial J(u),\dis \sum_{i \notin J_{2}}
(-\sum_{k=1}^{n} b_{k}) Z_{i}(u) \Big \rangle  &\leq& -c\biggr(
\sum_{i \notin J_{2}} \dis \frac{1}{\lambda_{i}^{\beta_{i}}}+
\sum_{i \notin J_{2}}  \dis \frac{|\nabla K(a_{i}) |}{\lambda_{i}}\biggr)\end{eqnarray} $$+O\Big( \sum_{k=1}^{q}\;\; \sum_{i\neq j\in
B_{k},\;i \notin J_{2} }\varepsilon_{ij}\Big)
+o
\Big(\dis\sum_{i=1}^{p} \frac{1}{\lambda_{i}^{\beta_{i}}}
\Big).$$

$$\mbox{Define }\hskip3.1cm W_{1}^{4}(u)=M_{1}\biggr(M_{1}\overline{Z}(u) +\overline{W}(u)\biggr)+\dis \sum_{i
\notin J_{2}} (-\sum_{k=1}^{n} b_{k}) Z_{i}(u).\hskip5cm$$
Using \eqref{38}, we  get
$$\Big \langle \partial J(u),W_{1}^{4}(u) \Big \rangle  \leq -c
\biggr( \dis \sum_{i=1}^{p} \dis
\frac{1}{\lambda_{i}^{\beta_{i}}}+ \dis \sum_{i=1}^{p} \dis
\frac{|\nabla K(a_{i}) |}{\lambda_{i}} + \dis \sum_{i \neq j}
\varepsilon_{i j} \biggl),$$
since $\dis \frac{1}{\lambda_{i}^{\beta_{i}}}=o\biggr(\dis\frac{1}{\lambda_{i_{k}}^{\beta_{i_{k}}}}\biggr)$ $\forall \; i\in B_{k}$ such that
$(i,i_{k})\not\in \widetilde{C}.$\\
The vector field $W_{1}$ in $V_{1}(p,\varepsilon)$ will be a convex combination of $W_{1}^{j}, j=1,...,4.$ From the definitions of $W_{1}^{j}, j=1,...,4$ the
only case where the maximum of the $\lambda_{i}$'s increase is when $a_{i}\in B(y_{l_{i}}, \rho)$,  $y_{l_{i}}\in \mathcal{K}^{+},$\;
$\forall\;i=1,...,p$,  with $y_{l_{i}}\neq y_{l_{j}}, \forall \; i \neq j.$ This conclude the proof of proposition \ref{3.2}.

\vskip0.2cm
\noindent\textbf{\textbf{Proof of proposition \ref{3.1}.}}
We divide the set
$V_{2}(p,\varepsilon)$ into five sets.
\begin{eqnarray*}
V^{1}_{2}(p,\varepsilon)&= &\Big\{u=\dis \sum_{i=1}^{p}\alpha_{i}
\delta_{a_{i} \lambda_{i}}\in V_{2}(p,\varepsilon),\;
y_{l_{i}}\neq y_{l_{j}} \; \forall i \neq j,\;  - \dis
\sum_{k=1}^{n}b_{k}(y_{l_{i}}) > 0,\;
 \\
&& \lambda_{i}|a_{i}-y_{l_{i}}|< \delta,\; \forall i=1,...,p
\mbox{ and } \rho(y_{l_{i}},...,y_{l_{p}})> 0 \Big\}.
\end{eqnarray*}
\begin{eqnarray*}
V^{2}_{2}(p,\varepsilon)&= &\Big\{u=\dis \sum_{i=1}^{p}\alpha_{i}
\delta_{a_{i} \lambda_{i}}\in V_{2}(p,\varepsilon),\;
y_{l_{i}}\neq y_{l_{j}} \; \forall i \neq j,\;  - \dis
\sum_{k=1}^{n}b_{k}(y_{l_{i}}) > 0,\;
 \\
&& \lambda_{i}|a_{i}-y_{l_{i}}|< \delta,\; \forall i=1,...,p
\mbox{ and } \rho(y_{l_{i}},...,y_{l_{p}})< 0 \Big\}.
\end{eqnarray*}
\begin{eqnarray*}
V^{3}_{2}(p,\varepsilon)&= &\Big\{u=\dis \sum_{i=1}^{p}\alpha_{i}
\delta_{a_{i} \lambda_{i}}\in V_{2}(p,\varepsilon),\;
y_{l_{i}}\neq y_{l_{j}} \; \forall i \neq
j,\;\lambda_{i}|a_{i}-y_{l_{i}}|< \delta,
 \\
&& \forall i=1,...,p, \mbox{and there exist } j \mbox{ (at least)
such that }- \dis \sum_{k=1}^{n}b_{k}(y_{l_{j}}) < 0 \Big\}.
\end{eqnarray*}
\begin{eqnarray*}
V^{4}_{2}(p,\varepsilon)&= &\Big\{u=\dis \sum_{i=1}^{p}\alpha_{i}
\delta_{a_{i} \lambda_{i}}\in V_{2}(p,\varepsilon),\;
y_{l_{i}}\neq y_{l_{j}} \; \forall i \neq j,
 \mbox{and there exist } j \mbox{ (at least)}\\
&& \mbox{such that }\lambda_{j}|a_{j}-y_{l_{j}}|\geq\dis
\frac{\delta}{2} \Big\}.
\end{eqnarray*}
\begin{eqnarray*}
V^{5}_{2}(p,\varepsilon)&= &\Big\{u=\dis \sum_{i=1}^{p}\alpha_{i}
\delta_{a_{i} \lambda_{i}}\in V_{2}(p,\varepsilon),\;
 \mbox{such that there exist } i\neq j
  \mbox{ satisfying }\\
  &&y_{l_{i}}= y_{l_{j}} \Big\}.
\end{eqnarray*}
We break up the proof into five steps. We construct an appropriate pseudo-gradient in each region and then glue up via convex combinations. Let $Z_1$ and $Z_2$ be two vector fields. A convex combination of $Z_1$ and $Z_2$ is given by $ \theta Z_1+ (1-\theta) Z_2$ where $\theta$ is cut-off function.\\

\noindent\underline{Step 1}: First, we
consider the case of $u= \dis \sum_{i=1}^{p}\alpha_{i}
\delta_{a_{i} \lambda_{i}}\in V_{2}^{1}(p,\varepsilon)$, we have
for any $i \neq j,\;|a_{i}-a_{j} |> \rho $ and therefore,
\begin{eqnarray*}
  \varepsilon_{i j} &=& \Big( \dis \frac{2}{(1- \cos d(a_{i},a_{j})) \lambda_{i}\lambda_{j}}  \Big)^{\frac{n-2\sigma}{2}}
  (1+o(1))\\
  &=& 2^{\frac{n-2\sigma}{2}} \dis \frac{G(a_{i},a_{j})}{(\lambda_{i}
  \lambda_{j})^{\frac{n-2\sigma}{2}}}(1+o(1) ).
\end{eqnarray*}
Where $G(a_{i},a_{j})$ is defined in (\ref{G}). Thus,
$$\lambda_{i}\dis \frac{\partial
\varepsilon_{i j}}{\partial \lambda_{i}} = -\dis
\frac{n-2\sigma}{2}2^{\frac{n-2\sigma}{2}} \dis
\frac{G(a_{i},a_{j})}{(\lambda_{i}
  \lambda_{j})^{\frac{n-2\sigma}{2}}}(1+o(1)).$$
Using proposition \ref{p3.2} with $\beta=n-2\sigma$ and the fact that
$\alpha_{i}^{\frac{4\sigma}{n-2\sigma}}
K(a_{i})J(u)^{\frac{n}{n-2\sigma}}=1+o(1)\; \forall i=1,...,p.$, we derive that
\begin{eqnarray*}
  \Big\langle \partial J(u), \alpha_{i} \lambda_{i} \dis \frac{\partial \delta_{i}}{\partial \lambda_{i}}
   \Big\rangle &=& \dis \frac{n-2\sigma}{2} \;J(u)^{1-\frac{n}{2}} \biggr[ \dis \frac{n-2\sigma}{n}\;
   \widetilde{c}_{1}\;
   \dis \frac{\sum_{i=1}^{p} b_{k}}{K(a_{i})^{\frac{n}{2\sigma}}} \;\dis\frac{1}{\lambda_{i}^{n-2\sigma}}\\
  &+&{c}_{1}\; 2^{\frac{n-2\sigma}{2}} \;\dis \sum_{i\neq j}\dis
  \frac{G(y_{l_{i}},y_{l_{j}})}{\Big( K(a_{i}) K(a_{j})\Big)^{\frac{n-2\sigma}{4\sigma}}}\frac{1}{
  \Big( \lambda_{i} \lambda_{j}
  \Big)^{\frac{n-2\sigma}{2}}}\biggl]\\
  \nonumber &+& o\biggr(\dis \sum_{i=1}^{p} \dis\frac{1}{\lambda_{i}^{n-2\sigma}}
 +\dis \sum_{i \neq j} \varepsilon_{i
j}\biggl) .
\end{eqnarray*}
Where $\widetilde{c}_{1}= c_{0}^{\frac{2n}{n-2\sigma}} \dis \int_{\dis
\mathbb{R}^{n}} \dis \frac{|(x_{1})|^{n-2\sigma}}{(1+|x |^{2})^{n}} dx
.$ Hence, using the fact that $\big|a_{i}-y_{l_{i}} \big|<\delta$,
$\delta$ very small, we get, \begin{eqnarray*}
 \Big \langle \partial J(u), \dis \sum_{i=1}^{p} \alpha_{i} Z_{i} \Big\rangle &\leq& -c \;\; ^{t} \Lambda \;M(y_{l_{1}},...,y_{l_{p}} ) \Lambda+o\biggr(\dis \sum_{i=1}^{p} \dis\frac{1}{\lambda_{i}^{n-2\sigma}}
 +\dis \sum_{i \neq j} \varepsilon_{i
j}\biggl) \\
  &\leq& - c \;\rho(y_{l_{1}},..., y_{l_{p}})\; |\Lambda |^{2}+o\biggr(\dis \sum_{i=1}^{p} \dis\frac{1}{\lambda_{i}^{n-2\sigma}}
 +\dis \sum_{i \neq j} \varepsilon_{i
j}\biggl),
\end{eqnarray*}
where $\Lambda= \; ^{t}\Big(
\frac{1}{\lambda_{1}^{\frac{n-2\sigma}{2}}},\ldots ,
\frac{1}{\lambda_{p}^{\frac{n-2\sigma}{2}}}  \Big)$. Here
$M(y_{l_{1}},..., y_{l_{p}})$ is defined in \eqref{4} and
$\rho(y_{l_{1}},...,y_{l_{p}})$ is the least eigenvalue of
$M(y_{l_{1}},...,y_{l_{p}})$. Using the fact that $\forall i \neq
j$, we have $\varepsilon_{i j} \leq \dis \frac{c}{(\lambda_{i}
\lambda_{j})^{\frac{n-2\sigma}{2}}}$, since $|a_{i}-a_{j} | \geq
\delta$, we then obtain

$$ \Big \langle \partial J(u), \dis \sum_{i=1}^{p} \alpha_{i} Z_{i}   \Big \rangle \leq -c \biggr(\dis \sum_{i=1}^{p} \dis\frac{1}{\lambda_{i}^{n-2\sigma}}
 +\dis \sum_{i \neq j} \varepsilon_{i
j}\biggl) .$$ In addition, $\forall i=1,...,p$, we have
$\lambda_{i} |a_{i}|< \delta  \; \Longrightarrow \; \dis
\frac{|\nabla K(a_{i}) |}{\lambda_{i}} \sim \dis
\frac{|(a_{i})_{k} |^{\beta-1}}{ \lambda_{i}} \leq \dis
\frac{c}{\lambda_{i}^{\beta}}$. Thus, we derive for $W^{1}_{2}:=
\dis \sum_{i=1}^{p} \alpha_{i} Z_{i}$
$$ \Big \langle \partial J(u), W^{1}_{2}  \Big \rangle \leq -c \biggr( \dis \sum_{i=1}^{p} \dis \frac{1}{\lambda_{i}^{n-2\sigma}}+
\dis \sum_{i=1}^{p} \dis \frac{|\nabla K(a_{i}) |}{\lambda_{i}} +
\dis \sum_{i \neq j} \varepsilon_{i j} \biggl). $$ \underline{Step
2}: Secondly, we study the case of $u= \dis \sum_{i=1}^{p}
\alpha_{i} \delta_{a_{i}\lambda_{i}} \in V_{2}^{2}(p,\varepsilon)
$. Let,\\ $e=(e_{i})_{i=1,...,p}$ an eigenvector associated to
$\rho(y_{l_{1}},...,y_{l_{p}})$ such that $|e |=1$ with $e_{i} >
0$ $ \forall i=1,...,p$. Let $\gamma
>0$ such that for any $x \in B(e,\gamma)= \{ y \in S^{p-1}\mbox{ s.t } |y-e |\leq \gamma
\}$, we have

$$^{t}x M(y_{l_{1}},...,y_{l_{p}})x \leq \dis
\frac{1}{2} \rho(y_{l_{1}},...,y_{l_{p}}).$$ Two cases may occur.\\

\noindent\underline{case 1:} $\dis \frac{\Lambda}{|\Lambda |} \in
B(e,\gamma )$, where $\Lambda =\; ^{t}\Big( \dis
\frac{1}{\lambda_{1}^{\frac{n-2\sigma}{2}}},\ldots , \dis
\frac{1}{\lambda_{p}^{\frac{n-2\sigma}{2}}} \Big) $. In this case, we
define $W_{2}^{2}= - \dis \sum_{i=1}^{p} \alpha_{i} Z_{i} $. As in
step 1, we find that,
$$ \Big \langle \partial J(u), W_{2}^{2}(u) \Big \rangle \leq -c \biggr( \dis \sum_{i=1}^{p} \dis \frac{1}{\lambda_{i}^{n-2\sigma}}
+ \dis \sum_{i=1}^{p} \dis \frac{|\nabla K(a_{i}) |}{\lambda_{i}}
+ \dis \sum_{i \neq j} \varepsilon_{i j} \biggl) .$$
\underline{case 2:} $\dis \frac{\Lambda}{|\Lambda |} \notin
B(e,\gamma )$. In this case, we define
$$ W_{2}^{2}= - \dis \frac{2}{n-2\sigma} |\Lambda | \dis \sum_{i=1}^{p} \alpha_{i} \lambda_{i}^{\frac{n}{2}} \Big[ \dis \frac{|\Lambda|\; e_{i}- \Lambda_{i}}
{| \Lambda |}- \dis \frac{\Lambda_{i} < |\Lambda|\;e - \Lambda,
\Lambda  > }{|\Lambda |^{3}} \Big]\dis \frac{\partial
\delta_{a_{i} \lambda_{i}}}{\partial \lambda_{i}}. $$ Using
proposition \ref{p3.2}, we find that
$$\Big \langle \partial J(u), W_{2}^{2}(u)   \Big \rangle = - c |\Lambda |^{2}  \dis \frac{\partial}{\partial t}\Big(\; ^{t}\Lambda(t) M\;
\Lambda (t) \Big)_{\dis /_{t=0}} + o\Big( \dis \sum_{i=1}^{p} \dis
\frac{1}{\lambda_{i}^{n-4}}\Big)+o\biggr(\dis \sum_{i \neq j}
\varepsilon_{i j}\biggl) .$$ Where $M= M(y_{l_{1}},...,
y_{l_{p}})$ and $\Lambda(t)= \dis \frac{(1-t) \Lambda + t |\Lambda
| e}{\Big|(1-t) \Lambda + t |\Lambda | e  \Big|} \Lambda $.
Observe that,
$$\; ^{t}\Lambda(t) M \Lambda(t)= \rho + \dis \frac{(1-t)^{2}}{\Big|(1-t) \Lambda + t |\Lambda | e   \Big|} \Big(\;  ^{t}\Lambda M \Lambda
- \rho | \Lambda|^{2}\Big).$$ Thus we obtain, $ \dis
\frac{\partial}{\partial t}\Big(\; ^{t}\Lambda(t) M\; \Lambda (t)
\Big)< -c $ and therefore we get,
$$ \Big \langle \partial J(u), W_{2}^{2}(u) \Big \rangle \leq -c \biggr( \dis \sum_{i=1}^{p} \dis \frac{1}{\lambda_{i}^{n-2\sigma}}
+ \dis \sum_{i=1}^{p} \dis \frac{|\nabla K(a_{i}) |}{\lambda_{i}}
+ \dis \sum_{i \neq j} \varepsilon_{i j} \biggl) .$$

\noindent\underline{Step 3}: Now, we deal with the case of $u= \dis
\sum_{i=1}^{p}\alpha_{i} \delta_{a_{i}\lambda_{i} } \in
V^{3}_{2}(p, \varepsilon) $.\\
Without loss of generality, we can assume that 1,...,q are the
indices which satisfy $- \dis  \sum_{k=1}^{n} b_{k}(y_{l_{i}})
<0\; \forall i=1,...,q$. Let,
$$\widetilde{W}_{2}^{1}= \dis
\sum_{i=1}^{q} - \alpha_{i} Z_{i}.$$  By proposition \ref{p3.2}
and \eqref{pp}, we obtain $$\big \langle
\partial J(u), \widetilde{W}_{2}^{1}(u) \Big\rangle  \leq -c \biggr(  \dis
\sum_{i=1}^{q}  \dis \frac{1}{\lambda_{i}^{n-2\sigma}}+ \dis \sum_{i
\neq j,\; 1\leq i\leq q} \varepsilon_{i j} \biggl) . $$ Set $$I=
\Big \{ i, 1 \leq i \leq p \;\mbox{ s.t }\; \lambda_{i} \leq \dis
\frac{1}{10} \dis \min_{1 \leq j\leq q} \lambda_{j}\Big \}.$$ It
is easy to see that, we can add to the above estimates all indices
$i$ such that $ i \notin I$. Thus
$$ \big \langle
\partial J(u), \widetilde{W}_{2}^{1}(u) \Big\rangle  \leq -c \biggr(  \dis
\sum_{i\notin I}  \dis \frac{1}{\lambda_{i}^{n-2\sigma}}+ \dis \sum_{i
\neq j,\; i \notin I} \varepsilon_{i j} \biggl) .$$ If $I \neq
\emptyset $, in this case, we write $u$ as follows
$$u= \dis
\sum_{i \in I} \alpha_{i} \delta_{a_{i} \lambda_{i}}+  \dis
\sum_{i \notin I} \alpha_{i} \delta_{a_{i} \lambda_{i}} =
u_{1}+u_{2}. $$ Observe that $u_{1}$ has to satisfy one of two
cases above that is $u_{1}\in V_{2}^{1}(\sharp I, \varepsilon)$ or
$u_{1}\in V_{2}^{2}(\sharp I, \varepsilon)$. Thus we can apply the
associated vector field which we will denote
$\widetilde{W}^{2}_{2}$. We then have
$$ \Big \langle \partial
J(u), \widetilde{W}^{2}_{2}(u) \Big \rangle  \leq -c \biggr(\dis
\sum_{i \in I}\dis \frac{1}{\lambda_{i}^{n-2\sigma}}+ \dis \sum_{i
\neq j,\; i \in I} \varepsilon_{i j}+\dis \sum_{i=1}^{p} \dis
\frac{|\nabla K(a_{i}) |}{\lambda_{i}}  \biggl) + O \biggr(\dis
\sum_{i \neq j,\; i \notin I} \varepsilon_{i j} \biggl).$$ Let in
this subset $W_{2}^{3}= \widetilde{W}^{1}_{2}+ m_{1}
\widetilde{W}^{2}_{2}$, $m_{1}$ be a small positive constant. We
get,
$$ \Big \langle \partial J(u), W_{2}^{3}(u) \Big \rangle \leq -c \biggr( \dis \sum_{i=1}^{p} \dis \frac{1}{\lambda_{i}^{n-2\sigma}}
+ \dis \sum_{i=1}^{p} \dis \frac{|\nabla K(a_{i}) |}{\lambda_{i}}
+ \dis \sum_{i \neq j} \varepsilon_{i j} \biggl) .$$

\noindent\underline{Step 4}: We consider her the case of $u= \dis
\sum_{i=1}^{p} \alpha_{i} \delta_{a_{i} \lambda_{i}} \in
V_{2}^{4}(p,\varepsilon). $\\
We order the $\lambda_{i}$'s in an increasing order, for sake of
simplicity, we can assume that
$\lambda_{1}\leq...\leq\lambda_{p}$. Let $\lambda_{i_{1}}= \inf \{
\lambda_{j} \mbox{ s.t } \lambda_{j}|a_{j} | \geq \delta\}$. For
$m_{1}>0$ small enough, we need to prove the following claim
$$ \Big \langle \partial J(u), (X_{i_{1}}- m_{1} Z_{i_{1}})(u)  \Big\rangle  \leq -c \biggr( \dis \sum_{i=i_{1}}^{p}
\dis \frac{1}{\lambda_{i}^{n-2\sigma}}+ \dis \sum_{j \neq i_{1}}
\varepsilon_{i_{1} j}  + \dis \sum_{i=1}^{p} \dis \frac{|\nabla
K(a_{i_{1}}) |}{\lambda_{i_{1}}} \biggl) .$$ Indeed, for $i \neq
j$, we have $| a_{i}-a_{j}|> \rho$, thus in proposition \ref{p3.3}
the term $\Big | \dis \frac{1}{\lambda_{i}} \dis \frac{\partial
\varepsilon_{i j}}{\partial(a_{i})_{k}}\; \Big|$ is very small
with respect $\varepsilon_{i j}$, hence,
\begin{eqnarray*}
\Big\langle \partial J(u), X_{i_{1}}(u) \Big\rangle &\leq& - \dis
\frac{c}{\lambda_{i_{1}}^{n-2\sigma}} \biggr( \dis
\int_{\dis\mathbb{R}^{n}} b_{k_{i_{1}}} \dis \frac{|x_{k_{i_{1}}}+
\lambda_{i_{1}}(a_{i_{1}})_{k_{i_{1}}} |^{\beta}}{\Big(1+
\lambda_{i_{1}}|(a_{i_{1}})_{k_{i_{1}}} |
\Big)^{\frac{\beta-1}{2}}}\dis \frac{x_{k_{i_{1}}}}{\Big( 1+
|x|^{2} \Big)^{n+1}}dx\biggl)^{2}\\&+& o\Big( \dis
\frac{1}{\lambda_{i_{1}}^{n-2\sigma}}
 + \dis \sum_{j \neq
i_{1}} \varepsilon_{i_{1j}} \Big).
\end{eqnarray*}
 If $i_{1}\in L_{1}$ in this
case $\delta \leq \lambda_{i_{1}}|a_{i_{1}} |\leq M_{1}$, using elementary calculation, we have
\begin{eqnarray}\label{73}
\biggl(\dis
\int_{\dis \mathbb{R}^{n} } b_{k_{i}} \dis \frac{|x_{k_{i}}+
\lambda_{i}(a_{1})_{k_{i}} |^{\beta}}{(
 1+ \lambda_{i} |(a_{1})_{k_{i}} | )^{\frac{\beta -1}{2}}} \dis \frac{x_{k_{i}}}{(1+
 |x|^{2})^{n}} dx  \biggr)^{2}\geq c >0.
 \end{eqnarray}
 Using
\eqref{73} , we get
\begin{equation}\label{X11}
 \Big \langle \partial J(u), X_{i_{1}}(u)  \Big\rangle  \leq -c \dis
\frac{1}{\lambda_{i_{1}}^{n-2\sigma}}+o\Big( \dis \sum_{j \neq i_{1}}
\varepsilon_{i_{1j}}  \Big) \leq -c \dis \sum_{i=i_{1}}^{p}\dis
\frac{1}{\lambda_{i}^{\beta}}+o\Big( \dis \sum_{j \neq i_{1}}
\varepsilon_{i_{1j} } \Big).
\end{equation}
From another part, we have by proposition \ref{p3.2} and
\eqref{pp},
\begin{equation}\label{X22}
\Big \langle \partial J(u), Z_{i_{1}}(u)  \Big\rangle  \leq -c
\dis \sum_{j \neq i_{1}} \varepsilon_{i_{1j}}+O \Big( \dis
\frac{1}{\lambda_{i_{1}}^{n-2\sigma}} \Big).
\end{equation}
Using \eqref{X11} and \eqref{X22} our claim
follows in this case.\\
If $i_{1}\in L_{2}$, using \eqref{eq4.1}, we find
\begin{eqnarray*}
\Big\langle \partial J(u), X_{i_{1}}(u) \Big\rangle &\leq& - c
\biggr( \dis \frac{1}{\lambda_{i_{1}}^{n-2\sigma}}+ \dis
\frac{|(a_{i_{1}})_{k_{i_{1}}}
|^{\beta-1}}{\lambda_{i_{1}}}\biggl)+o\Big( \dis \sum_{j \neq
i_{1}}
\varepsilon_{i_{1j}} \Big) \\
&\leq& -c \biggr(  \dis \sum_{i=i_{1}}^{p} \dis
\frac{1}{\lambda_{i}^{n-2\sigma}}+  \dis
\frac{|(a_{i_{1}})_{k_{i_{1}}} |^{\beta-1}}{\lambda_{i_{1}}}
\biggl)+o\Big( \dis \sum_{j \neq i_{1}} \varepsilon_{i_{1j}} \Big)
\end{eqnarray*}
and by proposition \ref{p3.2} and \eqref{eq4.1}, we have
$$ \Big\langle \partial J(u), - Z_{i_{1}}(u) \Big \rangle  \leq -c \dis \sum_{j \neq i_{1}}
\varepsilon_{i_{1j}} + O\Big(\dis \frac{|(a_{i_{1}})_{k_{i_{1}}}
|^{\beta-2}}{\lambda_{i_{1}}^{2}} \Big).$$ Now using \eqref{ab},
we obtain
\begin{eqnarray*}
  \Big\langle \partial J(u), (X_{i_{1}}-m_{1} Z_{i_{1}})(u) \Big \rangle  &\leq& -c \biggr( \dis \sum_{i=i_{1}}^{p} \dis
\frac{1}{\lambda_{i}^{n-2\sigma}} + \dis \sum_{j \neq i_{1}}
\varepsilon_{i_{1} j} +  \dis \frac{|(a_{i_{1}})_{k}
|^{\beta-1}}{\lambda_{i_{1}}} \biggl) \\
 &\leq&-c \biggr( \dis \sum_{i=i_{1}}^{p} \dis
\frac{1}{\lambda_{i}^{n-2\sigma}} + \dis \sum_{j \neq i_{1}}
\varepsilon_{i_{1} j} +  \dis \frac{|\nabla K(a_{i_{1}})
|}{\lambda_{i_{1}}} \biggl),
\end{eqnarray*}
since $|\nabla K(a_{i_{1}}) | \sim |(a_{i_{1}})_{k_{i}}
|^{\beta-1}$
hence our claim is valid.\\
 Now let, $$I=\Big\{ i, 1\leq i \leq p \;\; \mbox{ s.t }\; \lambda_{i}<
\dis \frac{1}{10} \lambda_{i_{1}} \Big\},$$ it is easy to see that
$$ \Big\langle \partial J(u), (X_{i_{1}}-m_{1} Z_{i_{1}})(u) \Big \rangle \leq -c \biggr( \dis \sum_{i\notin I} \dis
\frac{1}{\lambda_{i}^{n-2\sigma}} + \dis \sum_{j \neq i,\; i \notin I}
\varepsilon_{i j} +  \dis \frac{|\nabla K(a_{i_{1}})
|}{\lambda_{i_{1}}} \biggl).  $$ Furthermore,
using \eqref{eq4.1}, we have
$$\Big\langle \partial J(u), \Big(X_{i_{1}}-m_{1} Z_{i_{1}}+ \dis \sum_{i \notin I,\; i \in L_{2}}X_{i} \Big)(u) \Big \rangle \leq - c
\biggr( \dis \sum_{i\notin I} \dis \frac{1}{\lambda_{i}^{n-2\sigma}}+
\dis \sum_{i \notin I}\dis \frac{|\nabla K(a_{i}) |}{\lambda_{i}}
+ \dis \sum_{i \neq j,\; i \notin I} \varepsilon_{i j} \biggl)
$$
since for $i \notin I$ and $i \in L_{1}$ we have $\dis
\frac{|\nabla K(a_{i}) |}{\lambda_{i}} \leq \dis
\frac{c}{\lambda_{i}^{\beta}} .$\\
We need to add the remainder terms (if $I \neq \emptyset$). Let
$u_{1}= \dis \sum_{i \in I} \alpha_{i}\delta_{a_{i} \lambda_{i}}$,
$\forall i \in I$ we have $\lambda_{i} | a_{i}|<\delta$, thus
$u_{1}\in V_{2}^{j}(\sharp I, \varepsilon)$, $j=1$ or $2$ or $3$,
we can apply then the associated vector field which we will denote
$\widetilde{W}^{4}_{2}$. We then have
$$\Big\langle \partial J(u), \widetilde{W}^{4}_{2}\Big \rangle \leq -c \biggr( \dis \sum_{i \in I} \dis \frac{1}{\lambda_{i}^{n-2\sigma}}+
\dis \sum_{i \neq j,\; i,j \in I}\varepsilon_{i j} + \dis
\sum_{i\in I}  \frac{|\nabla K(a_{i}) |}{\lambda_{i}} \biggl)+ O
\Big(\dis \sum_{i \in I,\; j \notin I}\varepsilon_{i j}  \Big).$$
Let $W_{2}^{4}=X_{i_{1}}-m_{1}Z_{i_{1}} + \dis \sum_{i \notin I,\;
i\in L_{2}} X_{i} + m_{2}\widetilde{W}^{4}_{2} $, $m_{2}$ is
positive small enough, we get
$$\Big\langle \partial J(u), W^{4}_{2}(u)\Big \rangle \leq -c \biggr( \dis \sum_{i=1}^{p} \dis \frac{1}{\lambda_{i}^{n-2\sigma}}+
\dis \sum_{i=1}^{p}  \frac{|\nabla K(a_{i}) |}{\lambda_{i}} +\dis
\sum_{i\neq j}\varepsilon_{i j} \biggl).$$

\noindent\underline{Step 5}: We study now the case of $u= \dis
\sum_{i=1}^{p} \alpha_{i} \delta_{a_{i} \lambda_{i}} \in
V_{2}^{5}(p,\varepsilon).$\\
Let, $$B_{k}=\{j,\; 1\leq j\leq p\; \mbox{ s.t }\; a_{j} \in
B(y_{l_{k}},\rho) \}.$$ In this case, there is at least one
$B_{k}$ which contains at least two indices. Without loss of
generality, we can assume that $1,...,q$ are the indices such that
the set $B_{k}$, $1\leq k \leq q$ contains at least two indices.
We will decrease the $\lambda_{i}$'s for $i \in B_{k}$ with
different speed. For this purpose, let
$$\begin{array}{ccc}
  \chi : \mathbb{R}& \longrightarrow & \mathbb{R}^{+} \\
   t& \longmapsto &  \left\{%
\begin{array}{ll}
    0 & \hbox{  if    }| t|\;\leq \gamma'  \\\\
    1 & \hbox{  if    } | t|\;\geq 1. \\
\end{array}%
\right.
\end{array}$$
Where $\gamma'$ is a small constant. \\
For $j\in B_{k}$, set $\overline{\chi}\;(\lambda_{j})= \dis
\sum_{i \neq j,\; i \in B_{k}} \chi \; \Big(\dis
\frac{\lambda_{j}}{\lambda_{i}}\Big) $. Define
$$\widetilde{W}_{2}^{5}= - \dis \sum_{k=1}^{q}\; \dis \sum_{j \in B_{k}} \alpha_{j}\; \overline{\chi }\;(\lambda_{j})\; Z_{j}.$$
Using proposition \ref{p3.2} and \eqref{eq4.1}, we obtain
\begin{eqnarray*}
  \Big \langle \partial J(u),\widetilde{W}_{2}^{5}(u)   \Big \rangle & \leq & c \dis \sum_{k=1}^{q}
  \biggr[ \dis \sum_{i \neq j,\; j \in B_{k}}\overline{\chi}\;(\lambda_{j})  \lambda_{j} \dis \frac{\partial
    \varepsilon_{i j}}{\partial\lambda_{j}}+ \dis \sum_{ j \in B_{k},\; j\in L_{1}}\overline{\chi}\;(\lambda_{j})
    O\Big( \dis \frac{1}{\lambda_{j}^{n-2\sigma}} \Big)   \\
   &+&\dis \sum_{ j \in B_{k},\; j\in L_{2}}\overline{\chi}\;(\lambda_{j})
    O\Big(\dis \frac{|(a_{j})_{k_{i}}|^{\beta-2}}{\lambda_{j}^{2}}\Big)\biggl].
\end{eqnarray*}
For $j\in B_{k}$, with $k\leq q$, if
$\overline{\chi}\;(\lambda_{j})\neq 0 $, then there exists $i \in
B_{k}$ such that $\dis \frac{1}{\lambda_{j}^{n-2\sigma}}=
o(\varepsilon_{i j})$ (for $\rho$ small enough). Furthermore, for
$j \in B_{k}$, if $i \notin B_{k}$ (or $i \in B_{k}$ with
$\lambda_{i} \sim \lambda_{j}$), then we have  by \eqref{pp},

$$\lambda_{j} \dis \frac{\partial
    \varepsilon_{i j}}{\partial\lambda_{j}} \leq -c \;\varepsilon_{i
    j}\hbox{ and } \lambda_{i} \dis \frac{\partial
    \varepsilon_{i j}}{\partial\lambda_{i}} \leq -c \;\varepsilon_{i
    j}.$$ In the case where $i \in B_{k}$ with (assuming $\lambda_{i} <<
    \lambda_{j}$), we have $\overline{\chi}(\lambda_{j}) - \overline{\chi}(\lambda_{i}) \geq 1
    $. Thus
$$ \overline{\chi}\;(\lambda_{j})\;\lambda_{j} \dis \frac{\partial
    \varepsilon_{i j}}{\partial\lambda_{j}}+ \overline{\chi}\;(\lambda_{i})\;\lambda_{i} \dis \frac{\partial
    \varepsilon_{i j}}{\partial\lambda_{i}} \leq \lambda_{j} \dis \frac{\partial
    \varepsilon_{i j}}{\partial\lambda_{j}} \leq -c\; \varepsilon_{i j}. $$
Thus we obtain
\begin{eqnarray}
\nonumber \Big \langle \partial J(u),\widetilde{W}_{2}^{5}(u)
   \Big \rangle & \leq & - c \dis \sum_{k=1}^{q}  \; \dis \sum_{j \in B_{k}}\overline{\chi}\;(\lambda_{j})\Big(
   \dis \sum_{i \neq j}\varepsilon_{i j} + \dis \frac{1}{\lambda_{j}^{n-2\sigma}} \Big) \\
   \label{x25}
   &+&\dis \sum_{k=1}^{q}  \dis \sum_{j \in B_{k}, j \in
   L_{2}}\overline{\chi}\;(\lambda_{j}) O\Big(\dis \frac{|(a_{j})_{k_{i}} |^{\beta-2}}{\lambda_{j}^{2}}
    \Big).
\end{eqnarray}
We need to add the indices $j$, $j\in \; ^{C}\Big(
\bigcup_{K=1}^{q} B_{k}\Big) \bigcup \Big\{j \in B_{k} \;\mbox{
s.t }\; \overline{\chi}\; (\lambda_{j})=0  \Big\} .$ Let,
$$\lambda_{i_{0}}= \dis \inf\{\lambda_{i},\; i=1,...,p  \}.$$
We distinguish two cases.\\
\underline{case 1:} there exists $j$ such that
$\overline{\chi}\;(\lambda_{j}) \neq 0$ and $\lambda_{i_{0}} \sim
\lambda_{j}$, $\Big( \gamma' \leq \dis
\frac{\lambda_{i_{0}}}{\lambda_{j}} \leq 1 \Big)$, then we can
appear on the above estimate $-\dis
\frac{1}{\lambda_{i_{0}}^{n-2\sigma}}$ and therefore $- \dis
\sum_{i=1}^{p}\dis \frac{1}{\lambda_{i}^{n-2\sigma}}  $ and $- \dis
\sum_{k \neq r}\varepsilon_{k r}$. Thus we obtain
 $$ \Big \langle \partial J(u),\widetilde{W}_{2}^{5}(u)   \Big \rangle  \leq  - c
 \biggr( \dis
\sum_{i=1}^{p}\dis \frac{1}{\lambda_{i}^{n-2\sigma}} + \dis
\sum_{i\neq j}\varepsilon_{i j}  \biggl)+ O\biggr( \dis
\sum_{k=1}^{q}  \; \dis \sum_{j \in B_{k},\; j\in L_{2}}
\dis\frac{|(a_{j})_{k_{i}} |^{\beta-2}}{\lambda_{j}^{2}} \biggl).
$$
Now let,

$$W_{2}^{5}=\widetilde{W}_{2}^{5}+ m_{1} \dis
\sum_{i=1}^{p} X_{i} ,$$ using  the above estimates with
proposition \ref{p3.3} and \eqref{ab}, we obtain
$$\Big \langle \partial J(u),W_{2}^{5}(u)   \Big \rangle  \leq  - c
 \biggr( \dis
\sum_{i=1}^{p}\dis \frac{1}{\lambda_{i}^{n-2\sigma}} + \dis
\sum_{i\neq j}\varepsilon_{i j} + \dis \sum_{i=1}^{p} \dis
\frac{|\nabla K(a_{i}) |}{\lambda_{i}} \biggl).
$$
\underline{case 2:} For each $j\in B_{k}$, $1\leq k \leq q$ we
have $$\lambda_{i_{0}} << \lambda_{j} \;\; \Big(i.e.\; \dis
\frac{\lambda_{i_{0}}}{\lambda_{j}} < \gamma' \Big)\hbox{ or if
}\lambda_{i_{0}} \sim \lambda_{j} \mbox{ we have }
\overline{\chi}\;(\lambda_{j})=0 .$$ In this case we define

\begin{center}$ D= \Big[\Big\{ i,
\;\overline{\chi}\;(\lambda_{i})=0 \Big\}\bigcup\;\;
^{C}\Big(\bigcup_{k=1}^{q} B_{k}\Big) \Big] \bigcap \Big\{  i,\;
\dis \frac{\lambda_{i}}{\lambda_{i_{0}}}< \dis \frac{1}{\gamma'}
\Big\}$.
\end{center}
It is easy to see that $i_{0} \in D$ and if $i \neq j  \in \Big\{
i, \;\overline{\chi}\;(\lambda_{i})=0 \Big\}\bigcup\;\;
^{C}\Big(\bigcup_{k=1}^{q} B_{k}\Big)$ we have $a_{i}\in
B(y_{l_{i}},\rho)$ and $a_{j}\in B(y_{l_{j}}, \rho)$ with
$y_{l_{i}}\neq y_{l_{j}}$. Let, $$u_{1}= \dis \sum_{i \in D}
\alpha_{i} \delta_{a_{i} \lambda_{i}},$$ $u_{1}$ has to satisfy
one of the four subsets above, that is $u_{1}\in V^{j}_{2}(\sharp
I, \varepsilon)$ for $j=1,2,3$ or $4$. Thus we can apply the
associated vector field which we will denote $Y$ and we have the
estimate

$$\Big \langle \partial J(u),Y(u)   \Big \rangle \leq  -
c \biggr( \dis \sum_{i\in D}\dis \frac{1}{\lambda_{i}^{n-2\sigma}}
+\dis \sum_{i\in D} \dis \frac{|\nabla K(a_{i}) |}{\lambda_{i}}+
\dis \sum_{i\neq j,\; i,j\in D }\varepsilon_{i j} \biggl)
+O\biggr( \dis \sum_{i\in D,\; j\notin D }\varepsilon_{i j}
\biggl).$$ Observe in the above majoration we have the term $-
\dis \frac{1}{\lambda_{i_{0}}^{n-2\sigma}}$, thus we can make appear
$-\dis\sum_{i=1}^{p} \dis\frac{1}{\lambda_{i}^{n-2\sigma}} $. Now
concerning the term $-\dis \sum_{i\neq j}\varepsilon_{i j}$, if $i
\in D$ and $j \in \; ^{C}D$, observe that,$$^{C}D = \Big\{i, \;
\dis \frac{\lambda_{i}}{\lambda_{i_{0}}}
>\dis \frac{1}{\gamma'} \Big\} \cup \Big[\Big\{ i,
\;\overline{\chi}\;(\lambda_{i})\neq 0 \Big\} \cap \Big(
\cup_{k=1}^{q} B_{k}\Big) \Big],$$ we have two situations: either
$j \in \Big[\Big\{ i, \;\overline{\chi}\;(\lambda_{i})\neq0 \Big\}
\bigcap\Big(\cup_{k=1}^{q} B_{k}\Big) \Big]$, then we have
$-\varepsilon_{i j}$ in the estimates \eqref{x25} or $j \in
\Big\{i,\; \dis \frac{\lambda_{i}}{\lambda_{i_{0}}}> \dis
\frac{1}{\gamma'} \Big\}$, we can prove in this cases that
$|a_{i}-a_{j} |\geq \rho$. Thus
$$ \varepsilon_{i j} \leq \dis \frac{c}{(\lambda_{i} \lambda_{j})^{\frac{n-2\sigma}{2}}} <
\dis \frac{ c \gamma'^{\frac{n-2\sigma}{2}} }{(\lambda_{i_{0}}
\lambda_{i})^{\frac{n-2\sigma}{2}}}=o (\varepsilon_{i_{0} i}) \hskip1cm
\mbox{(for} \;\gamma' \;\mbox{ small enough)}.
$$
Thus we derive,
\begin{eqnarray*}
\Big \langle \partial J(u),(\widetilde{W}_{2}^{5}+ m_{1} Y)(u)
\Big \rangle  &\leq & - c\biggr( \dis \sum_{i\in D}   \dis \frac{|
\nabla K(a_{i})|}{\lambda_{i}}
 + \dis \sum_{i=1}^{p} \dis \frac{1}{\lambda_{i}^{n-2\sigma}} + \dis \sum _{i \neq j} \varepsilon_{i j}\biggl)\\
&+&\dis \sum_{K=1}^{q}  \; \dis \sum_{j \in B_{k},\; j\in
L_{2}}\overline{\chi}(\lambda_{j}) O\Big(
\dis\frac{|(a_{j})_{k_{i}} |^{\beta-2}}{\lambda_{j}^{2}}   \Big),
\end{eqnarray*}
and hence, by \eqref{ab}, we have
$$
\Big \langle \partial J(u),\Big(\widetilde{W}_{2}^{5}+ m_{1} Y+
m_{2} \dis \sum_{i=1,\; i \in L_{2}} X_{i}\Big)(u) \Big \rangle
\leq - c\biggr(  \dis \sum_{i=1}^{p} \dis
\frac{1}{\lambda_{i}^{n-2\sigma}} + \dis \sum _{i \neq j}
\varepsilon_{i j}+\dis \sum_{i=1}^{p} \dis \frac{| \nabla
K(a_{i})|}{\lambda_{i}} \biggl),
$$
for $m_{1}$ and $m_{2}$ two small positive constants. In this case
we denote $$W_{2}^{5}:= \widetilde{W}_{2}^{5}+ m_{1} Y+ m_{2} \dis
\sum_{i=1,\; i \in L_{2}} X_{i}.$$ The vector field $W_{2}$ in
$V_{2}(p,\varepsilon)$ will be a convex combination of
$W_{2}^{j},\; j=1,...,5$. This conclude the proof of proposition
\ref{3.1}.

\begin{cor}\label{c4.4}
Let $p\geq 1$. The critical points at infinity of $J$ in
$V(p,\varepsilon)$ correspond to $$(y_{l_{1}},...,y_{l_{p}})_{\infty}:=\dis \sum_{i=1}^{p} \dis \frac{1}{K(y_{l_{i}})^{\frac{n-2\sigma}{2}}} \;\delta_{(y_{l_{i}}, \infty)},$$
where $(y_{l_{1}},...,y_{l_{p}})\in \mathcal P^{\infty}$.
Moreover, such a critical point at infinity has an index equal to
$i(y_{l_{1}},...,y_{l_{p}})_{\infty} =p-1+\dis \sum_{i=1}^{p}n-\widetilde i(y) $.
\end{cor}

\section{Proof of Theorem \ref{TH2}}

Using corollary \ref{c4.4}, the only critical points at infinity associated to problem \eqref{1.3} correspond to $w_{\infty}=(y_{i_{1}},...,y_{i_{p}})\in\mathcal P^{\infty}$. We prove Theorem \ref{TH2} by contradiction. Therefore, we assume that equation \eqref{1.3} has no solution. For any $w_{\infty}\in\mathcal P^{\infty}$, let $c(w)_{\infty}$ denote the associated critical value at infinity. Here we choose to consider a simplified situation where for any $w_{\infty}\neq w_{\infty}^{'}$, $c(w)_{\infty}\neq c(w^{'})_{\infty}$ and thus order the $c(w)_{\infty}$'s, $w_{\infty}\in \mathcal P^{\infty}$ as

$$c(w_{1})_{\infty}<...<c(w_{k_{0}})_{\infty}. $$
For any $\overline{c}\in \mathbb{R}$, let $J_{\overline{c}}=\{ u \in \Sigma^{+}, \, J(u) \leq \overline{c} \}$. By using a deformation lemma (see \cite{BR1}), we know that if $c\,(\,w_{k-1}\,)_{\infty}<\,a\,<\,c\,(\,w_{k}\,)_{\infty}\,<\,b\,<\,c\,(\,w_{k+1}\,)_{\infty}$, then

\begin{eqnarray}\label{82}
J_{b}\simeq J_{a}\cup W_{u}^{\infty}(w_{k})_{\infty},
\end{eqnarray}
Here $W_{u}^{\infty}(w_{k})_\infty$ denote the unstable manifolds at infinity of $(w_k)_{\infty}$(see \cite{b2}) and $\simeq$ denotes retracts by deformation.\\
We apply the Euler-Poincar\'{e} characteristic of both sides of \eqref{82}, we find that

\begin{eqnarray}\label{48}
 \chi(J_{b})=\chi(J_{a})+(-1)^{i(w_{k})_{\infty}},
\end{eqnarray}
where ${i(w_{k})_{\infty}}$ denotes the index of the critical point at infinity $(w_{k})_{\infty}$. Let
$$b_{1}<c(w_{1})_{\infty}=\min_{u\in\Sigma^{+} } J(u)<b_{2}<c(w_{2})_{\infty}<...<b_{k_{0}}<c(w_{k_{0}})_{\infty}<b_{k_{0}+1}. $$
Since we have assumed that \eqref{1.3} has no solution, $J_{b_{k_{0}+1}}$ is a retard by deformation of $\Sigma^{+}$. Therefore $\chi(J_{b_{k_{0}+1}})=1,$
since $\Sigma^{+}$ is a contractible set. Now using \eqref{48}, we derive after recalling that $\chi(J_{b_{1}})=\chi(\emptyset)=0,$

\begin{eqnarray}\label{49}
1=\dis\sum_{j=1}^{k_{0}}(-1)^{i(w_{j})_{\infty}}.
\end{eqnarray}
So, if \eqref{49} is violated, then \eqref{1.3} has a solution. This complete the proof of Theorem \ref{TH2}.

\appendix

\section{Appendix A}

This appendix is devoted to some useful expansions of the gradient of $J$ near a potential critical points at infinity consisting of $p$ masses. Those propositions are proved under some technical estimates of the different integral quantities, extracted from \cite{b1} (with some change). In order to simplify the notations, in the remainder we write
$\delta_{i}$ instead of $\delta_{(a_{i},\lambda_{i})}$.

\begin{prop} \label{p3.2}
Assume that $K$ satisfies $(f)_{\beta}$, $1<\beta<n.$ For any
$U=\dis \sum_{j=1}^{p} \alpha_{j} \delta_{j}$ in
$V(p,\varepsilon)$, the following expansion hold
\begin{eqnarray}
 \nonumber
(i)\Big\langle \partial
J(U),\displaystyle\lambda_{i}\frac{\partial\displaystyle\delta_{i}}{\partial\displaystyle\lambda_{i}}\Big\rangle&=&-2c_{2}J(u)\displaystyle
\sum_{i\neq
j}\displaystyle\alpha_{j}\displaystyle\lambda_{i}\frac{\partial\displaystyle\varepsilon_{ij}}{\partial\displaystyle\lambda_{i}}
+o\biggr(\displaystyle \sum_{i\neq j}
\displaystyle\varepsilon_{ij}\biggr) +\displaystyle
o\biggr(\dis\frac{1}{\lambda_{i}}\biggr),
\end{eqnarray}where $c_{2}=c_{0}^{\frac{2n}{n-2\sigma}}\displaystyle
\int_{ \displaystyle \mathbb R^{n}}\frac{d\displaystyle
 y}{\big(1+|y|^{2}\big)^{\frac{n+2\sigma}{2}}}.$
\\
$(ii)$ If $a_{i}\in B(y_{j_{i}},\rho)$, $ y_{j_{i}}\in
\mathcal{K}$ and $\rho$ is a  positive constant small enough
, we have\\\\
$\Big \langle \partial J(U),\lambda_{i} \dis \frac{\partial
\delta_{i}}{\partial \lambda_{i}} \Big\rangle$
\begin{eqnarray}
 \nonumber
 &=&
  2J(u)\biggr[ -{c}_{2}
    \dis \sum_{j \neq i} \alpha_{j} \lambda_{i}\dis \frac{\partial \varepsilon_{i j}}{\partial
    \lambda_{i}}+\dis \frac{n-2\sigma}{2n}{c}_{0}^{\frac{2n}{n-2\sigma}} \beta \dis \frac{\alpha_{i}}{K({a_{i})}} \;\dis
  \frac{1}{\lambda_{i}^{\beta}}\dis \sum_{k=1}^{n}  b_{k}\\
  \nonumber&\times& \dis
  \int_{\dis \mathbb{R}^{n}} sign \Big(x_{k}+ \lambda_{i} (a_{i}-y_{j_{i}})_{k}
  \Big)\Big|x_{k}+ \lambda_{i} (a_{i}-y_{j_{i}})_{k}
  \Big|^{\beta-1} \dis \frac{x_{k}}{(1+| x|^{2} )^{n}} dx\\
 \label{1'}
   &+& o\Big( \dis \sum_{j \neq i} \varepsilon_{i j} +  \dis \sum_{j=1}^{p} \dis \frac{1}{\lambda_{j}^{\beta}}
    \Big)\biggl].
\end{eqnarray}
$(iii)$ Furthermore, if $\lambda_{i} |a_{i}-y_{j_{i}}|< \delta $, for $\delta$ very  small, we then have \\
\begin{eqnarray}
 \nonumber
\Big \langle  \partial J(U),\lambda_{i} \dis \frac{\partial
\delta_{i}}{\partial \lambda_{i}} \Big\rangle &=&
  2J(u)\biggr[\dis \frac{n-2\sigma}{2n}\;\beta\;{c}_{3}\;\dis \frac{\alpha_{i}}{K(a_{i})} \;
  \frac{\sum_{k=1}^{n}b_{k}}{\lambda_{i}^{\beta}}-{c}_{2}
    \dis \sum_{j \neq i} \alpha_{j} \lambda_{i}\dis \frac{\partial \varepsilon_{i j}}{\partial
    \lambda_{i}}\hskip1.5cm\\
    \label{eq22}
    &+&o\Big( \dis \sum_{j \neq i} \varepsilon_{i j} + \dis \sum_{j=1}^{p} \dis \frac{1}{\lambda_{j}^{\beta}}
    \Big)\biggl],
\end{eqnarray}
where ${c}_{3}= {c}_{0}^{\frac{2n}{n-2\sigma}} \dis \int_{\dis S^{n}}
\dis \frac{| x_{1} |^{\beta}}{(1+ |x |^{2})^{n}} dx$.\\
\end{prop}

\begin{prop} \label{p3.3}
Under condition $(f)_{\beta}$, $1<\beta < n,$ for each    $ U =\dis
\sum_{j=1}^{p}\alpha_{j} \delta_{j} \in V(p, \varepsilon)$, we
have
\begin{eqnarray}
 \nonumber
(i)\Big\langle
\partial
J(U),\frac{1}{\displaystyle\lambda_{i}}\frac{\partial\displaystyle\delta_{i}}{\partial
a_{i}}\Big\rangle&=&-c_{5}J(u)^{2}
\displaystyle\alpha_{i}^{\frac{n+2\sigma}{n-2\sigma}}\displaystyle
\frac{\nabla
K(a_{i})}{\displaystyle\lambda_{i}}+O\biggr(\displaystyle
\sum_{i\neq j}
\displaystyle \frac{1}{\lambda_{i}}\big|\frac{\partial \varepsilon_{ij} }{\partial a_{i}}\big|\biggr)\\
\nonumber&+&o\biggr(\displaystyle \sum_{i\neq j}
\displaystyle\varepsilon_{ij}+\frac{1}{\lambda_{i}}\biggr),
\end{eqnarray}
where $c_{5}=\displaystyle \int_{ \displaystyle \mathbb
R^{n}}\frac{d\displaystyle
 y}{\big(1+|y|^{2}\big)^{n}}$.\\
(ii) if $a_{i} \in B(y_{j_{i}},\rho)$, $y_{j_{i}}\in \mathcal{K}$, we have\\\\
$\Big\langle \partial
J(U),\displaystyle\frac{1}{\displaystyle\lambda_{i}}\frac{\partial\displaystyle\delta_{i}}{\displaystyle\partial
(a_{i})_{k}}\Big\rangle$=
\begin{eqnarray}\nonumber
&-&
2(n-2\sigma)c_{0}^{\frac{2n}{n-2\sigma}}\alpha_{i}^{\frac{n+2\sigma}{n-2\sigma}}J(u)^{2}
\frac{1}{\displaystyle\lambda_{i}^{\beta}}
\displaystyle\int_{\displaystyle\mathbb R^{n}}
b_{k}\big|x_{k}+\displaystyle\lambda_{i}(a_{i}-y_{j_{i}})_{k}\big|^{\beta}\frac{x_{k}}
{\displaystyle\big(1+\big|x\big|^{2}\big)^{n+1}}
\displaystyle dy \\
\nonumber & +&o\biggr(\displaystyle \sum_{i\neq j}
\displaystyle\varepsilon_{ij}\biggr)+\displaystyle
o\biggr(\displaystyle\sum_{i=1}^{p}\frac{1}{\lambda_{i}^{\beta}}\biggr)+O\biggr(\displaystyle
\sum_{i\neq j} \displaystyle
\frac{1}{\lambda_{i}}\big|\frac{\partial \varepsilon_{ij}
}{\partial a_{i}}\big|\biggr),
\end{eqnarray}
where $k=1,...,n$ and $(a_{i})_{k}$ is the $k^{th}$ component if
$a_{i}$ in some geodesic normal coordinates system.
\end{prop}

\begin{prop}\label{p24}
Let $n\geq 2$. Suppose that $K$ satisfies  $(f)_{\beta}$, with $1<\beta<n$. There exists $c > 0$ such that the following holds \\$ \|
\overline{v }\|\leq c \dis \sum_{i=1}^{p}\biggl[\dis
\frac{1}{\lambda_{i}^{\frac{n}{2}}} + \dis
\frac{1}{\lambda_{i}^{\beta}} + \dis \frac{| \nabla K(a_{i})|}{
\lambda_{i}} +\dis \frac{(\log
\lambda_{i})^{\frac{n+2\sigma}{2n}}}{\lambda_{i}^{\frac{n+2\sigma}{2}}} \biggr] + \;c\; \left\{%
\begin{array}{ll}
    \dis \sum_{k \neq r} \varepsilon_{k\; r}^{\frac{n+2\sigma}{2(n-2\sigma)}}\Big(\log \varepsilon_{k r}^{-1} \Big)^{\frac{n+2\sigma}{2n}}, & \hbox{ if } n \geq 3 \\
 \dis \sum_{k \neq r} \varepsilon_{k\; r}\Big(\log \varepsilon_{k r}^{-1} \Big)
    ^{\frac{n-2\sigma}{n}}\;\;\;\;\;, & \hbox{ if } n < 3. \\
\end{array}%
\right. $
\end{prop}

\end{document}